\documentclass[12pt]{amsart}
\usepackage{amsmath}
\usepackage{amssymb}
\usepackage{amsfonts}

\newtheorem{theorem}{Theorem}
\theoremstyle{plain}

\newtheorem{corollary}{Corollary}

\newtheorem{lemma}{Lemma}

\newtheorem{proposition}{Proposition}
\newtheorem{remark}{Remark}

\numberwithin{equation}{section}

\begin{document}
\title[Minimal surfaces in spheres and a Ricci-like condition]{Minimal surfaces in spheres
and a Ricci-like condition}
\author{Amalia-Sofia Tsouri and Theodoros Vlachos}
\address{Department of Mathematics, University of Ioannina, 45110 Ioannina,
Greece}
\email{tsourisofia93@gmail.com}
\email{tvlachos@uoi.gr}
\thanks{The first named author would like to acknowledge financial support by the General Secretariat for
Research and Technology (GSRT) and the Hellenic Foundation for Research and
Innovation (HFRI) Grant No: 133.}

\begin{abstract}
We deal with minimal surfaces in spheres that are locally isometric to a
pseudoholomorphic curve in a totally geodesic $\mathbb{S}^{5}$ in the nearly K{\"a}hler sphere
$\mathbb{S}^6$. Being locally isometric to a pseudoholomorphic curve in $\mathbb{S}^5$ turns out to be equivalent to the Ricci-like condition $\Delta\log(1-K)=6K,$ where $K$ is the Gaussian curvature of the induced metric.
Besides flat minimal surfaces in spheres, direct sums of surfaces
in the associated family of pseudoholomorphic curves in $\mathbb{S}^5$ do satisfy this Ricci-like condition. Surfaces in both classes are  exceptional surfaces. These are minimal surfaces whose all Hopf differentials are holomorphic, or equivalently the curvature ellipses have constant eccentricity up to the last but one.
Under appropriate global assumptions, we prove that minimal surfaces in spheres that satisfy this Ricci-like condition are indeed exceptional. Thus, the classification of these surfaces is reduced to the classification of
exceptional surfaces that are locally isometric to a pseudoholomorphic curve in
$\mathbb{S}^5.$ In fact, we prove, among other results,  that such exceptional surfaces in odd dimensional spheres are flat or direct sums of surfaces in the associated family of a pseudoholomorphic curve in $\mathbb{S}^5$.
\end{abstract}

\maketitle

\section{Introduction}

A fundamental problem in the study of minimal surfaces is to classify those surfaces
isometric to a given one. More precisely, the following question has been addressed by Lawson in \cite{Law}:

\begin{quotation}
\textit{Given a minimal surface $f\colon M\to\mathbb{Q}_c^n$ in a $n$-dimensional space form
of curvature $c$, what is the moduli space of all noncongruent minimal surfaces $\tilde{f}
\colon M\to\mathbb{Q}_c^{n+m},$ any $m$, which are isometric to $f.$}
\end{quotation}

Partial answers to this problem were provided by several authors. For instance, see \cite{Cal2, J, L, Law, M1, N, S, S1, V9,V08}. 

A classical result due to Ricci-Curbastro \cite{Ric} asserts that the Gaussian curvature
$K\le0$ of any minimal surface in $\mathbb{R}^3$ satisfies the so-called Ricci condition
\[
\Delta\log(-K)=4K,
\]
away from totally geodesic points, where $\Delta$ is the Laplacian operator of the surface 
with respect to the induced metric $ds^2.$ This condition is equivalent to the flatness of
the metric $d\hat{s}^2=(-K)^{1/2}ds^2.$ 
Conversely (see \cite{L0}), a metric on a simply connected $2$-dimensional Riemannian manifold with negative Gaussian curvature is realized on a minimal surface in $\mathbb{R}^3,$ if the Ricci
condition is satisfied. Hence, the Ricci condition is a necessary and sufficient condition for
a $2$-dimensional Riemannian manifold to be locally isometric to a minimal surface in
$\mathbb{R}^3.$

Lawson \cite{L} studied the above problem for minimal surfaces in a Euclidean
space that are isometric to minimal surfaces in $\mathbb{R}^3.$ Using the Ricci condition
and the holomorphicity of the Gauss map, he classified all minimal surfaces in
$\mathbb{R}^n$ that are isometric to a minimal surface in $\mathbb{R}^3.$ Calabi
\cite{Cal2} obtained a complete description of the moduli space of all noncongruent
minimal surfaces in $\mathbb{R}^n$ which are isometric to a given holomorphic curve in
$\mathbb{C}^n.$

The aforementioned problem turns out to be more difficult for minimal surfaces in spheres. The difficulty arises from the fact that their Gauss map is merely harmonic, while for
minimal surfaces in the Euclidean space it is holomorphic. The classification problem of
minimal surfaces in spheres that are isometric to minimal surfaces in the sphere
$\mathbb{S}^3$ was raised by Lawson in \cite{L}, where he stated a conjecture that is still
open. This conjecture has been only confirmed for certain classes of minimal surfaces in
spheres (see \cite{N, S, S1, V9,V08}). It is worth noticing that a surface is locally isometric to a minimal surface in 
$\mathbb{S}^3$ if its Gaussian curvature 
satisfies the spherical  Ricci condition
$$
\Delta\log(1-K)=4K.
$$

A distinguished class of minimal surfaces in spheres is the one of \textit{pseudoholomorphic
curves} in the nearly K{\"a}hler sphere $\mathbb{S}^6,$ that was introduced by Bryant
\cite{Br} and has been widely studied (cf. \cite{BVW, H, EschVl}). The pseudoholomorphic curves in
$\mathbb{S}^6$ are nonconstant smooth maps from a Riemann surface into the
nearly K{\"a}hler sphere $\mathbb{S}^6,$ whose differential is complex linear with respect to the almost complex structure of 
$\mathbb{S}^6$ that is induced from the multiplication of the Cayley numbers. 

In analogy with Calabi's work \cite{Cal2}, we consider the problem of classifying minimal
surfaces in spheres that are isometric to pseudoholomorphic curves in the nearly K{\"a}hler sphere $\mathbb{S}^6.$
More precisely, in the present paper we focus on the following problem:
\textit{Classify minimal surfaces in spheres that are locally isometric to pseudoholomorphic curves
in a totally geodesic  $\mathbb{S}^5\subset\mathbb{S}^6.$} The case of  pseudoholomorphic curves
that are substantial in $\mathbb{S}^6$ requires a different treatment and will be the subject of a forthcoming paper.

A characterization of Riemannian metrics that
arise as induced metrics on pseudoholomorphic curves in $\mathbb{S}^5$ was given in
\cite{H, EschVl}. In fact, the Gaussian curvature $K\leq1$ of 
a pseudoholomorphic curve in $\mathbb{S}^5$ satisfies the condition
\[
\Delta\log(1-K)=6K, \tag{$\ast$}
\]
away from totally geodesic points, where $\Delta$ is the Laplacian operator of the induced
metric $ds^2.$ This condition is equivalent to the flatness of the metric $d\hat
{s}^2=(1-K)^{1/3}ds^2.$ Conversely, any two-dimensional Riemannian manifold $(M,ds^2),$ with Gaussian
curvature $K<1$, that satisfies the Ricci-like condition ($\ast$) can be locally isometrically
immersed as a pseudoholomorphic curve in $\mathbb{S}^5.$ 
Thus the classification of minimal surfaces in spheres that are locally isometric to a
pseudoholomorphic curve in $\mathbb{S}^5\subset\mathbb{S}^6$ is equivalent to the
classification of those surfaces whose induced metrics satisfy the condition $(\ast).$

Obviously flat minimal surfaces in spheres satisfy the condition $(\ast).$
These surfaces were classified in \cite{K2} and lie in odd dimensional spheres. Another class of minimal surfaces in odd dimensional spheres that satisfy the 
condition $(\ast)$ is constructed in the following way. Let $g_{\theta}, 0\leq \theta<\pi$, be the associated family of a simply connected pseudoholomorphic curve $g\colon M\to
\mathbb{S}^{5}.$ We consider the surface $\hat g\colon M\to\mathbb{S}^{6m-1}$ defined by
\begin{equation}
\hat g=a_{1}g_{\theta _{1}}\oplus \cdots\oplus a_{m}g_{\theta _{m}},  \tag{$\ast \ast$}
\end{equation}
where $a_{1},\dots\,,a_{m}$ are any real numbers with $\sum_{j=1}^{m}a_{j}^{2}=1,$
$0\leq \theta _{1}<\cdots<\theta_{m}<\pi,$ and $\oplus $ denotes the
orthogonal sum with respect to an orthogonal decomposition of the Euclidean space
$\mathbb{R}^{6m}.$  It is easy to see that $\hat g$ is minimal and isometric to $g.$ 

It would be interesting to know whether there exist other minimal surfaces in spheres whose induced metrics satisfy
the condition $(\ast),$ besides the flat ones and surfaces of the type $(\ast \ast).$ 

We prove that minimal surfaces given by $(\ast\ast)$ belong to the class of exceptional
surfaces that were studied in \cite{V08, V16}. These are minimal surfaces
whose all Hopf differentials are holomorphic, or equivalently all curvature ellipses of any order have constant eccentricity up to the last but one. This class of surfaces  contains the superconformal ones.  

It is then natural to investigate whether any minimal surface that
satisfies the Ricci-like condition $(\ast)$ is an exceptional surface. 
We are able to prove that minimal surfaces in spheres that satisfy the condition $(\ast)$ are exceptional
under appropriate global assumptions. In view of this result, the study of minimal surfaces
in spheres that satisfy the condition $(\ast)$ is reduced to the class of exceptional
surfaces.

In fact, we prove that besides flat minimal surfaces in odd dimensional spheres, the only
simply connected exceptional surfaces that satisfy the condition $(\ast)$ are of the type $(\ast\ast).$ Furthermore, we show that compact minimal surfaces in 
$\mathbb{S}^n, 4\le n\le7,$ that are not homeomorphic to the torus, cannot be locally isometric to
a pseudoholomorphic curve in $\mathbb{S}^5,$ unless $n=5.$ 
Moreover, we prove that, under certain assumptions, there are no
minimal surfaces in even dimensional spheres that satisfy the condition 
$(\ast).$

It is worth noticing that a necessary and sufficient condition for a $2$-dimensional
Riemannian manifold to be locally isometric to a minimal Lagrangian surface in 
$\mathbb{C}P^2$ (see \cite[Theorem 3.8]{EGT}) is that its induced metric satisfies the condition $(\ast).$ Thus our results apply to minimal surfaces in spheres that are locally
isometric to minimal Lagrangian surfaces in $\mathbb{C}P^2.$

The paper is organized as follows: In Section 2, we fix the notation and give some
preliminaries. In Section 3, we recall the notion of Hopf differentials and some results
about exceptional surfaces. In Section 4, we give some basic facts about absolute value type
functions, a notion that was introduced in \cite{EGT,ET}. In section 5, we discuss
pseudoholomorphic curves in $\mathbb{S}^{5}$ and give some useful properties for minimal
surfaces that satisfy the Ricci-like condition ($\ast$). In Section 6, we study the class of
surfaces of type ($\ast\ast$), we compute their Hopf differentials and show that they are
indeed exceptional. Section 7 is devoted to minimal surfaces that are exceptional and satisfy
the condition $(\ast).$ There we prove that, besides flat minimal surfaces in odd dimensional spheres, the only
simply connected exceptional surfaces that satisfy the condition $(\ast)$ are given
by $(\ast\ast).$ In the last section, we prove our global results.

\section{Preliminaries}

In this section, we collect several facts and definitions about minimal surfaces in spheres.
For more details on these facts we refer to \cite{DF} and \cite{DV2k15}. 
\vspace{1,5ex}

Let $f\colon M\to\mathbb{S}^n$ denote an isometric immersion of a two-dimensional
Riemannian manifold. The $k^{th}$\emph{-normal space} of $f$ at $x\in M$ for
$k\geq 1$ is defined as
$$
N^f_k(x)={\rm span}\left\{\alpha^f_{k+1}(X_1,\ldots,X_{k+1}):X_1,\ldots,X_{k+1}\in
T_xM\right\},
$$
where 
$$
\alpha^f_s\colon TM\times\cdots\times TM\to N_fM,\;\; s\geq 3, 
$$
denotes the symmetric tensor called the $s^{th}$\emph{-fundamental form} given
inductively by
$$
\alpha^f_s(X_1,\ldots,X_s)=\left(\nabla^\perp_{X_s}\cdots\nabla^\perp_{X_3}
\alpha^f(X_2,X_1)\right)^\perp
$$
and $\alpha^f\colon TM\times TM\to N_fM$ stands for the standard second fundamental 
form of $f$ with values in the normal bundle. Here, $\nabla^{\perp}$ denotes the induced
connection in the normal bundle $N_fM$ of $f$ and $(\,\cdot\,)^\perp$ means taking the
projection onto the orthogonal complement of $N^f_1\oplus\cdots\oplus N^f_{s-2}$ in
$N_fM.$ If $f$ is minimal, then ${\rm dim}N^f_k(x)\le2$ for all $k\ge1$ and any $x\in M$ (cf. \cite{DF}).

A surface $f\colon M\to\mathbb{S}^n$ is called \emph{regular} if for each $k$ the
subspaces $N^f_k$ have constant dimension and thus form normal subbundles. Notice
that regularity is always verified  along connected components of an open dense subset of
$M.$ 

Assume that an immersion $f\colon M\to\mathbb{S}^n$ is minimal and substantial.
By the latter, we mean that $f(M)$ is not contained in any totally geodesic submanifold of $\mathbb{S}^n.$ In this case, the normal bundle
of $f$ splits along an open dense subset of $M$ as
$$
N_fM=N_1^f\oplus N_2^f\oplus\dots\oplus N_m^f,\;\;\; m=[(n-1)/2],
$$
since all higher normal bundles  have rank two except possible the last one that has rank
one if $n$ is odd; see \cite{Ch} or \cite{DF}. Moreover, if $M$ is oriented, then an
orientation is induced on each plane subbundle $N_s^f$  given by the ordered base
$$
\alpha^f_{s+1}(X,\ldots,X),\;\;\;\alpha^f_{s+1}(JX,\ldots,X),
$$
where $0\neq X\in TM$. 

If $f\colon M\to\mathbb{S}^n$ is a minimal surface, then at $x\in M$ and for each
$N_r^f$, $1\leq r\leq m$, the \emph{$r^{th}$-order curvature ellipse}
$\mathcal{E}^f_r(x)\subset N^f_r(x)$ is defined  by
$$
\mathcal{E}^f_r(x) = \left\{\alpha^f_{r+1}(Z^{\varphi},\ldots,Z^{\varphi})\colon\,
Z^{\varphi}=\cos\varphi Z+\sin\varphi JZ\;\mbox{and}\;\varphi\in[0,2\pi)\right\},
$$
where $Z\in T_xM$ is any vector of unit length.

A substantial regular surface $f\colon M\to\mathbb{S}^{n}$ is called
\emph{$s$-isotropic} if it is  minimal and  at any $x\in M$ the  curvature ellipses
$\mathcal{E}^f_r(x)$ contained in all two-dimensional  $N^f_r$$\,{}^{\prime}$s 
are circles for any $1\le r\le s.$

The $r$-th \textit{normal curvature} $K_{r}^{\perp}$ of $f$ is defined by
\begin{equation*}
K_{r}^{\perp}={\frac{2}{\pi}}{\hbox {Area}}(\mathcal{E}^f\sb r).
\end{equation*}
If $\kappa _{r}\geq \mu_{r}\geq 0$ denote the length of the semi-axes of the  curvature ellipse
$\mathcal{E}^f\sb r,$ then
\begin{equation}\label{elipsi}
K_{r}^{\perp}=2\kappa _{r}\mu _{r}.
\end{equation}

The \textit{eccentricity} $\varepsilon\sb r$ of the  curvature ellipse $\mathcal{E}^f\sb r$ is given by
\begin{equation*}
\varepsilon\sb r=\frac{\left(\kappa^{2}_{r}-\mu^{2}_{r}\right)^{1/2}}{\kappa_{r}},
\end{equation*}
where $\left(\kappa^{2}_{r}-\mu^{2}_{r}\right)^{1/2}$ is the distance from the center to a focus,
and can be thought of as a measure of how far $\mathcal{E}^f\sb r$ deviates from being a
circle.

The $a$-\textit{invariants} (see \cite{V16}) are the functions
\begin{equation*}
a^{\pm}_{r}= \kappa _{r}{\pm}  \mu _{r}= \left(2^{-r} \Vert \alpha^f_{r+1}
\Vert^{2} \pm K_{r}^{\perp}\right)^{1/2}.
\end{equation*}
These functions determine the geometry of the $r$-th curvature ellipse.

Denote by $\tau^o_f$ the index of the last plane bundle, in the orthogonal decomposition
of the normal bundle. Let $\{e_1,e_2\}$ be a local tangent orthonormal frame and
$\{e_\alpha\}$ be a local orthonormal frame of the normal bundle such that
$\{e_{2r+1},e_{2r+2}\}$ span $N_r^f$ for any $1\le r\le\tau^o_f$ and $e_{2m+1}$
spans the line bundle $N^f_{m+1}$ if $n=2m+1.$ For any $\alpha=2r+1$ or $\alpha=2r+2,$ we set 
$$
h_1^{\alpha}=\langle
\alpha^f_{r+1}(e_1,\dots,e_1),e_{\alpha}\rangle,{\ }h_2^{\alpha}=\langle
\alpha^f_{r+1}(e_1,\dots,e_1,e_2),e_{\alpha}\rangle.
$$
Introducing the complex valued functions
$$
H_{\alpha}=h_1^{\alpha}+ih_2^{\alpha}\;\;\text{for any}\;\;\alpha=2r+1\;\;\text{or}\;\;\alpha=2r+2,
$$
it is not hard to verify that the $r$-th normal
curvature is given by
\begin{equation}\label{prwthsxeshden}
K_r^{\perp}=i\left(H_{2r+1}{\overline{H}_{2r+2}}-{\overline{H}_{2r+1}}
H_{2r+2}\right).
\end{equation}

The length of the $(r+1)$-th fundamental form $\alpha^f_{r+1}$ is given by
\begin{equation}\label{deuterhsxeshden}
\Vert \alpha^f_{r+1}\Vert ^2=2^r\big(|{H_{2r+1}}|^2+|{H_{2r+2}}|
^2\big),
\end{equation}
or equivalently (cf. \cite{A})
\begin{equation}\label{si}
\Vert \alpha^f_{r+1}\Vert ^2=2^r(\kappa_r^2+\mu_r^2).
\end{equation}

\smallskip
Each plane
subbundle $N_r^f$ inherits a Riemannian connection from that of the normal bundle. Its
\textit{intrinsic curvature} $K^*_r$ is given by the following proposition (cf. \cite{A}).

\begin{proposition}\label{5}
\textit{The intrinsic curvature $K_r^{\ast}$ of each plane subbundle $N_{r}^f$
of a minimal surface} $f\colon M\to \mathbb{S}^{n}$ \textit{is given by} 
\begin{equation*}
K_{1}^{\ast}=K_1^{\perp}-{\frac{\Vert \alpha^f_3\Vert ^2}{2K_1^{\perp}}} 
\;\; \text{and}\;\; K_r^{\ast}={\frac{K_r^{\perp}}{(K_{r-1}^{\perp})^2}}{\frac{
\Vert \alpha^f_{r}\Vert ^{2}}{2^{r-2}}}-{\frac{\Vert \alpha^f_{r+2}\Vert ^2}{
2^rK_r^{\perp}}}\;\;\text{for}\;\;2\leq r\leq \tau_f^o.
\end{equation*}
\end{proposition}

Let $f\colon M\to\mathbb{S}^n$ be a minimal isometric immersion. If $M$ is
simply connected, there exists a one-parameter \emph{associated family} of minimal
isometric immersions $f_\theta\colon M\to\mathbb{S}^n,$ where  $\theta\in\mathbb{S}^1=[0,\pi).$ To see this, for each
$\theta\in\mathbb{S}^1$ consider the orthogonal parallel tensor field 
$$
J_{\theta}=\cos\theta I+\sin\theta J,
$$
where $I$ is the identity endomorphism of the tangent bundle and $J$ is the complex
structure of $M$ induced by
the metric and the orientation.  Then, the symmetric section $\alpha^f(J_\theta\cdot,
\cdot)$ of the bundle $\text{Hom}(TM\times TM,N_f M)$ satisfies the Gauss, Codazzi and
Ricci equations, with respect to the same normal connection; see \cite{DG2} for details. 
Therefore, there exists a minimal isometric immersion  $f_{\theta}\colon M\to
\mathbb{S}^n$ whose second fundamental form is
\begin{equation}
\label{sff}
\alpha^{f_{\theta}}(X,Y)=T_\theta\alpha^f(J_{\theta}X,Y),
\end{equation}
where $T_\theta\colon N_fM\to N_{f_{\theta}}M$ is the parallel 
vector bundle isometry that identifies the normal subspaces
$N_s^f$ with $N_s^{f_\theta}$, $s\geq 1.$

\section{Hopf differentials and Exceptional surfaces}

Let $f\colon M\to\mathbb{S}^n$ be a minimal surface. The complexified tangent bundle
$TM\otimes \mathbb{C}$ is decomposed into the eigenspaces $T^{\prime}M$ and $T^{\prime \prime}M$ of the complex structure
$J$,  corresponding to the
eigenvalues $i$ and $-i.$  The $(r+1)$-th fundamental form $\alpha^f_{r+1}$, which takes values in
the normal subbundle $N_{r}^f$, can be complex linearly extended to $TM\otimes
\mathbb{C}$ with values in the complexified vector bundle $N_{r}^f\otimes \mathbb{C}$
and then decomposed into its $(p,q)$-components, $p+q=r+1,$ which are tensor
products of $p$ differential 1-forms vanishing on $T^{\prime \prime}M$ and $q$ differential 1-forms vanishing
on $T^{\prime}M.$ The minimality of $f$ is equivalent to the vanishing of the
$(1,1)$-part of the second fundamental form. Hence, the $(p,q)$-components of $\alpha^f_{r+1}$
vanish unless $p=r+1$ or $p=0,$ and consequently for a local complex coordinate $z$
on $M$, we have the following decomposition 
\begin{equation*}
\alpha^f_{r+1}=\alpha_{r+1}^{(r+1,0)}dz^{r+1}+\alpha_{r+1}^{(0,r+1)}d\bar{z}^
{r+1},
\end{equation*}
where 
\begin{equation*}
\alpha_{r+1}^{(r+1,0)}=\alpha^f_{r+1}(\partial,\dots,\partial),\;\;\alpha_{r+1}^{(0,r+1)}=\overline{\alpha_{r+1}^{(r+1,0)}}\;\;\;\text{and}\;\;\;\partial ={\frac{1}{2}}\big({\frac{\partial}{\partial x}}-i{\frac{\partial}{\partial y}}\big).
\end{equation*}

The \textit{Hopf differentials} are the differential forms (see \cite{V})
\begin{equation*}
\Phi _{r}=\langle \alpha_{r+1}^{(r+1,0)},\alpha_{r+1}^{(r+1,0)}\rangle dz^{2r+2}
\end{equation*}
of type $(2r+2,0),r=1,\dots,[(n-1)/2],$ where $\langle \cdot,\cdot\rangle$ denotes the extension of the usual Riemannian metric
of $\mathbb{S}^n$ to a complex bilinear form. These forms are
defined on the open subset where  the minimal surface is regular and are independent of the choice of coordinates, while
$\Phi _{1}$ is globally well defined.

Let $\{e_1,e_2\}$ be a local orthonormal frame in the tangent bundle.
It will be convenient to use complex vectors, and we put
\begin{equation*}
\text{ }E=e_1-ie_2\;\;
\text{and}\;\; \phi =\omega _{1}+i\omega _2,
\end{equation*}
where $\{\omega_1,\omega_2\}$ is the dual frame. We choose a local complex
coordinate $z=x+iy$ such that $\phi =Fdz.$

From the definition of Hopf differentials, we easily obtain 
\begin{equation*}
\Phi _{r}={\frac{1}{4}}\left({\overline{H}_{2r+1}^2}+{\overline{H}_{2r+2}^2}
\right) \phi^{2r+2}.
\end{equation*}

Moreover, using (\ref{prwthsxeshden}) and (\ref{deuterhsxeshden}), we find that
\begin{equation}\label{what}
\left\vert \langle \alpha_{r+1}^{(r+1,0)},\alpha_{r+1}^{(r+1,0)}\rangle \right\vert
^2=\frac{F^{2r+2}}{2^{2r+4}}\left(\Vert \alpha^f_{r+1}\Vert
^4-4^r(K_r^{\perp})^2\right).
\end{equation}
Thus, the zeros of $\Phi _r$ are precisely the points where the $r$-th curvature ellipse $
\mathcal{E}^f\sb r$ is a circle. Being $s$-isotropic is equivalent to $\Phi_r=0$ for any
$1\le r\le s.$

The Codazzi equation implies that $\Phi _{1}$ is always holomorphic (cf. \cite{Ch,ChW}).
Besides $\Phi_1$, the rest Hopf differentials are not always holomorphic. The following
characterization of the holomorphicity of Hopf differentials was given in \cite{V08}, in
terms of the eccentricity of curvature ellipses of higher order. 

\begin{theorem}\label{ena}
Let $f\colon M\to \mathbb{S}^n$ be a minimal surface. Its Hopf differentials
$\Phi _{2},\dots,\Phi_{r+1}$ are holomorphic if and only if the higher curvature ellipses
have constant eccentricity up to order $r.$
\end{theorem}

A minimal surface in $\mathbb{S}^n$ is called \textit{$r$-exceptional} if all Hopf
differentials up to order $r+1$ are holomorphic, or equivalently if all higher curvature
ellipses up to order $r$ have constant eccentricity. A minimal surface in $\mathbb{S}^n$
is called \textit{exceptional} if it is $r$-exceptional for $r=[(n-1)/2-1].$ 
This class of minimal surfaces may be viewed as the next simplest to superconformal ones.
In fact, superconformal minimal surfaces are indeed exceptional, for superconformal
minimal surfaces are characterized by the fact that all Hopf differentials vanish up to the
last but one, which is equivalent to the fact that all higher curvature ellipses are circles up to the last but one. There is an abundance of exceptional surfaces.
Besides flat minimal surfaces, we show in Section 6  that minimal surfaces of the type $(\ast\ast)$ are indeed
exceptional.

We recall some results for exceptional surfaces proved in \cite{V08}, that will be used in
the proofs of our main results.

\begin{proposition}\label{3i}
Let $f\colon M\to\mathbb{S}^n$ be an $(r-1)$-exceptional surface. At regular
points the following hold:

(i) For any $1\leq s\leq r-1,$ we have
\begin{equation*}
\Delta \log \left\Vert \alpha_{s+1}\right\Vert ^2=2\big((s+1)K-K_s^{\ast}\big),
\end{equation*}
where $\Delta $ is the Laplacian operator with respect to the induced metric $ds^{2}.$

(ii) If $\Phi _{r}\neq 0$\textit{, then}
\begin{equation*}
\Delta \log \left(\left\Vert \alpha_{r+1}\right\Vert^2+2^rK_r^{\perp}\right) =2\big((r+1)K-K_r^{\ast}\big)
\end{equation*}
and 
\begin{equation*}
\Delta \log \left(\left\Vert \alpha_{r+1}\right\Vert^2-2^rK_r^{\perp}\right) =2\big((r+1)K+K_r^{\ast}\big).
\end{equation*}

(iii) If $\Phi _{r}=0$\textit{, then}
\begin{equation*}
\Delta \log \left\Vert \alpha_{r+1}\right\Vert^2=2\big((r+1)K-K_r^{\ast}\big).
\end{equation*}

(iv) The intrinsic curvature of the $s$-th
normal bundle $N_s^f$\textit{\ is} $K_{s}^{\ast}=0$ 
if $1\leq s\leq r-1$ and $\Phi _s\neq 0.$
\end{proposition}

We will need the following proposition which was proved in \cite{V08}.
\begin{proposition}\label{neoksanaafththfora}
Let $f\colon M\to \mathbb{S}^n$ be an $r$-exceptional surface. Then the set $L_0$,
where $f$ fails to be regular, consists of isolated points and all $N_s^f$'s and the Hopf
differentials $\Phi_s$'s extend smoothly to $L_0$ for any $1\leq s\leq r.$ 
\end{proposition}

\section{Absolute value type functions}

For the proof of our results, we shall use the notion of absolute value type functions
introduced in \cite{EGT,ET}. A smooth complex valued function $p$ defined on a
Riemann surface is called of \textit{holomorphic type} if locally $p=p_0p_1,$ where $p_0$
is holomorphic and $p_1$ is smooth without zeros. A function $u\colon M\to\lbrack 0,+
\infty)$ defined on a Riemann surface $M$ is called of \textit{absolute value type} if there
is a function $p$ of holomorphic type on $M$ such that $u=|p|.$

The zero set of such a function on a connected compact oriented surface $M$ is either
isolated or the whole of $M$, and outside its zeros the function is smooth. If $u$ is a
nonzero absolute value type function, i.e., locally $u=|t_0|u_1$, with $t_0$ holomorphic,
the order $k\ge1$ of any point $p\in M$ with $u(p)=0$ is the order of $t_0$ at $p.$ Let
$N(u)$ be the sum of the orders for all zeros of $u.$ Then $\Delta\log u$ is bounded on
$M\smallsetminus\left\{u=0\right\}$ and its integral is computed in the following lemma
that was proved in \cite{EGT,ET}.

\begin{lemma}\label{forglobal}
Let $(M,ds^2)$ be a compact oriented two-dimensional Riemannian manifold with area
element $dA.$

(i) If $u$ is an absolute value type function on $M,$ then 
\begin{equation*}
\int_{M}\Delta \log udA=-2\pi N(u).
\end{equation*}

(ii) If $\Phi $ is a holomorphic symmetric $(r,0)$-form on $M,$ then either $\Phi =0$ or
$N(\Phi)=-r\chi (M),$ where $\chi (M)$ is the Euler-Poincar\'{e} characteristic of $M.$
\end{lemma}

The following lemma, that was proved in \cite{N}, provides a sufficient condition for a
function to be of absolute value type.
\begin{lemma}\label{dena}
Let $D$ be a plane domain containing the origin with coordinate $z$ and $u$ be a real
analytic nonnegative function on $D$ such that $u(0)=0.$ If $u$ is not identically zero
and $\log u$ is harmonic away from the points where $u=0$, then $u$ is of absolute
value type and the order of the zero of $u$ at the origin is even.
\end{lemma}

\section{Pseudoholomorphic curves in $\mathbb{S}^5$}

It is known that the multiplicative structure on the Cayley numbers $\mathbb{O}$ can be
used to define an almost complex structure $J$ on the sphere $\mathbb{S}^6$ in
$\mathbb{R}^7.$ This almost complex structure is not integrable but it is nearly
K{\"a}hler. A \textit{pseudoholomorphic curve} \cite{Br} is a nonconstant smooth map
$g\colon M\to\mathbb{S}^6$ from a Riemann surface $M$ into the nearly K{\"a}hler
sphere $\mathbb{S}^6,$ whose differential is complex linear.

It is known \cite{Br, EschVl} that any pseudoholomorphic curve $g\colon M\to\mathbb{S}
^6$ is $1$-isotropic. The nontotally geodesic pseudoholomorphic curves in $\mathbb{S}^6$
are  2-isotropic and substantial in $\mathbb{S}^6$, substantial in $\mathbb{S}^6$ but not 2-isotropic, or substantial in a totally geodesic $\mathbb{S}^5\subset\mathbb{S}^6.$

The following theorem \cite{EschVl} provides a characterization of Riemannian metrics
that arise as induced metrics on pseudoholomorphic curves in $\mathbb{S}^5.$
\begin{theorem}
Let $(M,ds^2)$ be a simply connected Riemann surface, and let $K\leq 1$ be its Gaussian curvature and $\Delta$ its Laplacian operator. Suppose that the function $1-K$ is of absolute value type. Then there exists an isometric pseudoholomorphic curve $g\colon M\to\mathbb{S}^5$ if and only if
\[
\Delta\log(1-K)=6K.\tag{$\ast$}
\]
In fact, up to translations with elements of $G_2$, that is the set $Aut(\mathbb{O})\subset SO_7,$ there is precisely one associated family of such maps.
\end{theorem}
The above result shows that a minimal surface in a sphere is locally isometric to a
pseudoholomorphic curve in $\mathbb{S}^5$ if its Gaussian curvature satisfies the 
condition $(\ast)$ at points where $K<1$ or equivalently if the metric
$d\hat{s}^2=(1-K)^{1/3}ds^2$ is flat.

The following lemma is fundamental for our proofs.

\begin{lemma}\label{avtf}
Let $f\colon (M,ds^2)\to\mathbb{S}^n$ be a nontotally geodesic minimal surface. If
$(M,ds^2)$ satisfies the Ricci-like condition $(\ast)$, at points with Gauss curvature
$K<1$, then the function $1-K$ is of absolute value type with isolated zeros of even order.
Moreover, if $M$ is compact and $p_j,j=1,\dots,m,$ are the isolated zeros of \,$1-K$ with
corresponding order $\mathrm{ord}_{p_j}(1-K)=2k_j,$ then we have
\begin{equation}\label{denginetai}
\sum_{j=1}^{m}k_j=-3\chi (M),
\end{equation}
where $\chi (M)$ is the Euler-Poincar\'{e} characteristic of $M.$ In particular, $M$ cannot be homeomorphic to the sphere $\mathbb{S}^2.$
\end{lemma}

\begin{proof}
Let $M_{0}$ be the set of points where $K=1.$ The open subset $M\smallsetminus M_0$ is dense on $M,$ since minimal surfaces in spheres are real analytic. Around each point 
$p_0\in M_0,$ we choose a local complex coordinate $z$ such that $p_0$
corresponds to $z=0$ and the induced metric is written as $ds^2=F|dz|^2.$ The Gaussian curvature $K$ is given by 
$$
K=-\frac{2}{F}\partial \bar{\partial}\log F.
$$

Moreover, the condition $(\ast)$ is equivalent to 
$$
4\partial \bar{\partial}\log (1-K)=6KF.
$$ 
Thus we have 
$$
\partial \bar{\partial}\log \left((1-K)F^3\right) =0.
$$ 

According to Lemma \ref{dena}, the function $1-K$ is of absolute value type with isolated
zeros $p_j,j=1,\dots,m,$ and corresponding order $\mathrm{ord}_{p_j}(1-K)=2k_{j}.$
Then, (\ref{denginetai}) follows from Lemma \ref{forglobal}(i) and the condition
$(\ast).$
\end{proof}

Let $g\colon M\to\mathbb{S}^5$ be a pseudoholomorphic curve and let
$\xi\in N_fM$ be a smooth unit vector field that spans the extended line bundle $N_2^g$ over the isolated set of points where $f$ fails to be regular  (see Proposition \ref{neoksanaafththfora}).   The
surface $g^*\colon M\to\mathbb{S}^5$ defined by $g^*=\xi$ is called the \textit{polar surface} of $g.$ It has been proved in \cite[Corollary 3]{V16} that the surfaces $g$ and $g^*$ are
congruent.

\section{A class of minimal surfaces that are locally isometric to a pseudoholomorphic
curve in $\mathbb{S}^5$}

The aim of this section is to study a class of minimal surfaces that are exceptional, nonflat and
locally isometric to a pseudoholomorphic curve in $\mathbb{S}^5.$

Let $g\colon M\to\mathbb{S}^5$ be a simply connected pseudoholomorphic curve with
Gaussian curvature $K<1,$ with respect to the induced metric $\langle\cdot,\cdot \rangle=ds^2,$ and let 
$g_\theta,\theta\in\mathbb{S}^1=[0,\pi),$ be its associated family.

Take 
$$
\bold{a}=(a_1,\dots,a_m)\in\mathbb{S}^{m-1}\subset\mathbb{R}^m 
\;\; \text{with}\;\; \prod\limits_{j=1}^{m}a_j\neq0
$$
and 
$$
{\pmb{\theta}}=(\theta_1,\dots,\theta_m)\in\mathbb{S}^1\times\cdots\times
\mathbb{S}^1, \;\; \text{where} \;\; 0\leq \theta _1<\cdots<\theta_m<\pi.
$$

We consider the map $\hat{g}=g_{\bold{a},\pmb{\theta}}\colon M\to\mathbb{S}^
{6m-1}\subset\mathbb{R}^{6m}$ defined by
\[
\hat{g}=g_{\bold{a},{\pmb{\theta}}}=a_1g_{\theta_{1}}\oplus\cdots
\oplus a_mg_{\theta_m},
\]
where $\oplus$ denotes the orthogonal sum with respect to an orthogonal decomposition
of $\mathbb{R}^{6m}$. Its differential is given by 
\[
d\hat{g}=a_{1}dg_{\theta_1}\oplus\cdots\oplus a_mdg_{\theta_m}.
\]
It is obvious that $\hat{g}$ is an isometric immersion. We can easily see that the second
fundamental form of the surface $\hat{g}$ is given by
\[
\hat{\alpha}_2(X,Y)=\sum\limits_{j=1}^ma_j\alpha^{g_{\theta_j}}(X,Y),\,\,
X,Y\in TM,
\]
which implies that $\hat{g}$ is minimal.

The following proposition provides several properties for the above class of minimal
surfaces.
\begin{proposition}\label{hola}
For any simply connected pseudoholomorphic curve $g\colon M\to\mathbb{S}^5,$ the
minimal surface $\hat{g}\colon M\to\mathbb{S}^{6m-1}$ is substantial and 
isometric to $g.$  Moreover, it is an exceptional surface
and the following hold:

(i) The length of its $(s+1)$-th fundamental form is given by
\begin{equation}\label{alphaf}
\left\Vert \hat{\alpha}_{s+1}\right\Vert^2=
\begin{cases}
\,\hat{b}_s(1-K)^{s/3} &\text{ \,if\, }
s\equiv 0\;\mathrm{mod}\; 3, \\[2mm]
\,\hat{b}_s(1-K)^{(s+2)/3}& \text{ \,if\, }
s\equiv 1\;\mathrm{mod}\; 3, \\[2mm]
\,\hat{b}_s(1-K)^{(s+1)/3} &\text{ \,if\, }
s\equiv 2\;\mathrm{mod}\; 3,
\end{cases}
\end{equation}
 for any $1\le s\le 3m-1,$ where $\hat{b}_s$ are positive numbers.

(ii) Its $s$-th normal curvature is given by
\begin{equation}\label{curvaturef}
\hat{K}^{\perp}_s=
\begin{cases}
\,\hat{c}_s\,(1-K)^{s/3}&\text{ \,if\, }
s\equiv 0\;\mathrm{mod}\; 3, \\[2mm]
\,\hat{c}_s\,(1-K)^{(s+2)/3}&\text{ \,if\, }
s\equiv 1\;\mathrm{mod}\; 3, \\[2mm]
\,\hat{c}_s\,(1-K)^{(s+1)/3}&\text{ \,if\, }
s\equiv 2\;\mathrm{mod}\; 3,
\end{cases}
\end{equation}
for any $1\le s< 3m-1,$ where $\hat{c}_s$ are positive numbers.

(iii) Its $s$-th Hopf differential  is given by
\begin{equation*}\label{hopff}
\hat{\Phi}_s=
\begin{cases}
\,\hat{d}_s\Phi^{(s+1)/3}&\text{ \,if\, }
s\equiv 2\;\mathrm{mod}\; 3, \\[1mm]
\,0&\text{ \,\,otherwise,}
\end{cases}
\end{equation*}
for any $1\le s\le 3m-1,$ where $\hat{d}_s\in\mathbb{C}$ and $\Phi$ is the second Hopf differential of $g.$
\end{proposition}

\begin{proof}
We consider a local orthonormal frame $\{e_1,e_2\}$ in the  tangent bundle away from totally geodesic points of $g.$ Moreover, we choose a local orthonormal frame field $\{\xi_1, \xi_2,\xi_3\}$ in the normal bundle of $g$ such that
\begin{equation*}
\xi_{1}=\frac{\alpha(e_1,e_1)}{\left\Vert \alpha(e_1,e_1)\right\Vert},\,\,\,
\xi_{2}=\frac{\alpha(e_1,e_2)}{\left\Vert \alpha(e_1,e_2)\right\Vert}.
\end{equation*}

From \cite[Lemma 5]{V16} it follows that $h_1^3=\kappa,$ $h_2^3=0,$
$h_1^4=0$ and $h_2^4=\kappa,$ where $\kappa$ is the radius of the first
circular curvature ellipse. Hence $H_{3}=\kappa$ and $H_{4}=i\kappa.$ Moreover, we
have that $h_2^5=0$ and $h_1^5=\kappa.$
Therefore, it follows that
\begin{equation*}
\langle\nabla_{e_1}^{\perp}\xi_1,\xi_3\rangle=1,\,\,\,
\langle\nabla_{e_1}^{\perp}\xi_2,\xi_3\rangle=0,
\end{equation*}
\begin{equation*}
\langle\nabla_{e_2}^{\perp}\xi_1,\xi_3\rangle=0,\,\,\,
\langle\nabla_{e_2}^{\perp}\xi_2,\xi_3\rangle=-1,
\end{equation*}
or equivalently
\begin{equation}\label{apeiroeths}
\langle\nabla_{\overline{E}}^{\perp}\xi_{3},\xi_{1}-i\xi_{2}\rangle=0,\,\,\,
\langle\nabla_{\overline{E}}^{\perp}\xi_{3},\xi_{1}+i\xi_{2}\rangle=-2,
\end{equation}
where $E=e_1-ie_2.$

In order to show that the minimal surface $\hat{g}$ is substantial, it is sufficient to prove that
\begin{equation}
\sum\limits_{j=1}^m a_j\langle g_{\theta_j},w_j\rangle=0\label{paragwgisi}
\end{equation}
for $(w_1,\dots,w_m)\in\mathbb{R}^{6m}=\mathbb{R}^6\oplus\cdots\oplus\mathbb{R}^6$ implies that $w_j=0$ for any $j=1,\dots,m.$

Assume to the contrary that $w_j\neq0$ for all $j=1,\dots,m.$ Differentiating
(\ref{paragwgisi}) we obtain
\begin{equation}\label{dgneo}
\sum\limits_{j=1}^m a_j\langle dg_{\theta_j},w_j\rangle=0,
\end{equation}
and
\[
\sum\limits_{j=1}^{m}a_j\left\langle \alpha^{g_{\theta_j}}, w_j\right\rangle=0.
\]
Using (\ref{sff}), we have that 
\[
\sum\limits_{j=1}^ma_j\left\langle T_{\theta_j}\alpha^g(J_{\theta_j}\overline{E},
\overline{E}), w_j\right\rangle=0,
\]
where  $T_{\theta_{j}}\colon N_gM\to N_{g_{\theta_j}}M$ is a parallel vector bundle
isometry. Since $J_{\theta}\overline{E}=e^{-i\theta}\overline{E}$, it follows that
\begin{equation}\label{Ttheta}
\sum\limits_{j=1}^{m}a_je^{-i\theta_j}\left\langle T_{\theta_j}(\xi_1+i\xi_2),
w_j\right\rangle=0.
\end{equation}
Differentiating with respect to $\overline{E}$, and using the Weingarten formula,
we obtain
\[
\sum\limits_{j=1}^{m}a_je^{-i\theta_j}\left\langle\nabla^\perp_{\overline{E}}
T_{\theta_j}(\xi_1+i\xi_2), w_j\right\rangle=\sum\limits_{j=1}^{m}a_je^{-i\theta_j}
\left\langle dg_{\theta_j}\circ A_{T_{\theta_j}(\xi_1+i\xi_2)}(\overline{E}),w_j\right
\rangle,
\]
where $A_{T_{\theta_j}\eta}$ is the shape operator of ${g_{\theta_j}}$ with respect to its normal
direction $T_{\theta_j}\eta.$ It follows from (\ref{sff}) that
\[
A_{T_{\theta_j}(\xi_1+i\xi_2)}=e^{i\theta_j}A_{\xi_1+i\xi_2}.
\]
This and \eqref{dgneo} yield
\[
\sum\limits_{j=1}^{m}a_je^{-i\theta_j}\left\langle T_{\theta_j}\left(\nabla^\perp_
{\overline{E}}\left(\xi_1+i\xi_2\right)\right), w_j\right\rangle=0.
\]
Using \eqref{apeiroeths} and \eqref{Ttheta}, the above is written as
\[
\sum\limits_{j=1}^{m}a_je^{-i\theta_j}\left\langle T_{\theta_j}\xi_3,w_j\right\rangle=0,
\]
or equivalently 
\[
\sum\limits_{j=1}^{m}a_je^{-i\theta_j}\langle g^*_{\theta_j},w_j\rangle=0,
\]
where $g^*_{\theta_j}=T_{\theta_j}\xi_3$ is the polar surface of $g_{\theta_j}.$ This
is equivalent to
\[
\sum\limits_{j=1}^{m}a_j\cos\theta_j\langle g^*_{\theta_j},w_j\rangle=0
\;\;\text{and}\;\;\sum\limits_{j=1}^{m}a_j\sin\theta_j\langle g^*_{\theta_j},
w_j\rangle=0.
\]
Eliminating $\langle g^*_{\theta_m},w_m\rangle$, we can easily see that
\begin{equation}\label{g*w}
\sum\limits_{j=1}^{m-1}a_j\langle g^*_{\theta_j},w_j^{(m)}\rangle=0,
\end{equation}
where $w_j^{(m)}=\sin(\theta_m-\theta_j)w_j\neq0.$ Using the fact that the polar
surface of $g_{\theta_j}$ is congruent to $g_{\theta_j}$ (cf. \cite[ Lemma 11]{DV} or
\cite[Corollary 3]{V16}) and arguing as for \eqref{g*w}, we have that
\[
\sum\limits_{j=1}^{m-2}a_j\langle g_{\theta_j},w_j^{(m-1)}\rangle=0,
\]
where $w_j^{(m-1)}=\sin(\theta_{m-1}-\theta_j)w_j^{(m)}.$ Arguing as before, we
inductively obtain that 
\[
a_1\langle g_{\theta_1},w\rangle=0 \text{\,\, or \,\,} a_1\langle g^*_{\theta_1},
w\rangle=0
\]
for a vector $w\in\mathbb{R}^6\smallsetminus\{0\},$ which is a contradiction.

\medskip
\textit{Claim.}
We now claim that the higher fundamental forms of $\hat{g}$ are given by
\begin{equation}\label{1}
\hat{\alpha}_s(\overline{E},\dots,\overline{E})=
\begin{cases}
\,\kappa^{s/3}\sum_{j=1}^{m}c_j^sg_{\theta_j}&\text{ \ if \ }s\equiv 0
\;\mathrm{mod}\;6,\\[2mm]
\,\kappa^{(s-1)/3}\sum_{j=1}^{m}c_j^sdg_{\theta_j}(\overline{E})
&\text{ \ if \ }s\equiv 1\;\mathrm{mod}\; 6, \\[2mm]
\,\kappa^{(s+1)/3}\sum_{j=1}^{m}c_j^sT_{\theta_j}(\xi_1+i\xi_2)
&\text{ \ if \ }s\equiv 2\;\mathrm{mod}\; 6, \\[2mm]
\,\kappa^{s/3}\sum_{j=1}^{m}c_j^sT_{\theta_j}\xi_3&\text{ \ if \ }s\equiv
3\;\mathrm{mod}\; 6, \\[2mm]
\,\kappa^{(s-1)/3}\sum_{j=1}^{m}c_j^sT_{\theta_j}(\xi_1-i\xi_2)
&\text{ \ if \ }s\equiv 4\;\mathrm{mod}\; 6, \\[2mm]
\,\kappa^{(s+1)/3}\sum_{j=1}^{m}c_j^sdg_{\theta_j}(E)
&\text{ \ if \ }s\equiv 5\;\mathrm{mod}\; 6,
\end{cases}
\end{equation}
where  the complex vectors $\bold{C}_s=(c_1^s,\dots,c_m^s)\in\mathbb{C}^m\smallsetminus
\{0\}$, $2\le s\le 3m$ satisfy the
following orthogonality conditions, with respect to the standard Hermitian product $(\cdot,\cdot)$
on $\mathbb{C}^m$:
\begin{equation}\label{2}
(\bold{C}_{t},\overline{\bold{C}}_{t'})=0 \,\,\, \text{if} \,\,\, t\equiv1\;\mathrm{mod}\; 6
\,\,\text{and}\,\, t'\equiv5\;\mathrm{mod}\; 6, \,\,\,\text{or}\,\,
t\equiv2\;\mathrm{mod}\; 6\,\,\text{and}\,\, t'\equiv4\;\mathrm{mod}\; 6,
\end{equation}
\begin{equation}
(\bold{C}_{t},\overline{\bold{C}}_{t'})=0=(\bold{C}_{t},\bold{C}_{t'})\,\,\text{if} \,\,
t\neq t' \,\,\text{and} \,\, t,t'\equiv0\;\mathrm{mod}\; 6, \,\,\,\text{or}\,\,t, t'\equiv3\;\mathrm{mod}
\; 6,
\end{equation} 
\begin{equation}
(\bold{C}_{t},\bold{C}_{t'})=0=(\overline{\bold{C}}_{t},\overline{\bold{C}}_{t'})
\text{ if\,\, } t\neq t'  \,\,\text{and} \,\, t,t'\equiv1\;\mathrm{mod}\; 6, \,\,\,\text{or} \,\, t, t'\equiv2\;
\mathrm{mod}\; 6, \\
\end{equation}
\begin{equation*}
\,\,\,\,\,\,\,\,\,\,\,\,\,\,\,\,\,\,\,\,\,\,\,\,\,\,\,\,\,\,\,\,\,\,\,\,\,\,\,\,\,\,\,\,\,\,\,\,\text{or} \,\, t, t'\equiv4\;\mathrm{mod}\; 6,  \,\,\text{or} \,\, t, t'\equiv5\;\mathrm{mod}\; 6 
\end{equation*}
and
\begin{equation}\label{6}
(\bold{C}_{t},\bold{a})=0 \text{ if\,\, }t\equiv0,1,5\;\mathrm{mod}\; 6.
\end{equation}
In particular, these complex vectors are defined inductively by
\begin{equation}\label{celeosksananea}
\bold{C}_{s+1}=
\begin{cases}
\,\bold{C}_{s}-\sum\limits_{t\equiv1\;\mathrm{mod}\; 6}^{s}
\frac{(\bold{C}_{s},
 \bold{C}_{t})}{\left\Vert\bold{C}_{t}\right\Vert^{2}}\bold{C}_{t}
-\sum\limits_{t\equiv5\;\mathrm{mod}\; 6}^{s}\frac{(\bold{C}_{s},
\overline{\bold{C}}_{t})}{\left\Vert\bold{C}_{t}\right\Vert^{2}}\overline{\bold{C}}_{t}
&\text{\,\,\,if $s\equiv 0\;\mathrm{mod}\; 6, $}\\[4mm]
\,2T_{\pmb{\theta}}\bold{C}_{s}
&\text{\,\,\,if $s\equiv 1\;\mathrm{mod}\; 6,$}\\[4mm]
\,2\bold{C}_{s}
&\text{\,\,\,if $s\equiv 2\;\mathrm{mod}\; 6,$}\\[3mm]
\,-\bold{C}_{s}+\sum\limits_{t\equiv2\;\mathrm{mod}\; 6}^{s}\frac{(\bold{C}_{s},
\overline{\bold{C}}_{t})}{\left\Vert\bold{C}_{t}\right\Vert^{2}}\overline{\bold{C}}_{t}
+\sum\limits_{t\equiv4\;\mathrm{mod}\; 6}^{s}\frac{(\bold{C}_{s},
\bold{C}_{t})}{\left\Vert\bold{C}_{t}\right\Vert^{2}}\bold{C}_{t}
&\text{\,\,\,if $s\equiv 3\;\mathrm{mod}\; 6,$ }\\[4mm]
\,-2T_{\pmb{\theta}}\bold{C}_{s}
&\text{\,\,\,if $s\equiv 4\;\mathrm{mod}\; 6,$}\\[4mm]
\,-2\bold{C}_{s}
&\text{\,\,\,if $s\equiv 5\;\mathrm{mod}\; 6,$}
\end{cases}
\end{equation}
where  $\bold{C}_1=\bold{a}$ and 
$T_{\pmb{\sigma}}\colon\mathbb{C}^m\to\mathbb{C}^m$ denotes the unitary
transformation given by
$$
T_{\pmb{\sigma}}\bold{u}=(u_{1}e^{-i\sigma_1},\dots,u_{m}e^{-i
\sigma_{m}}),\;\; \bold{u}=(u_{1},\dots,u_{m})\in\mathbb{C}^{m}
$$ 
for any  ${\pmb{\sigma}}=(\sigma_{1},\dots,\sigma_{m})\in\mathbb{R}^{m}.$ It is worth noticing that \eqref{1} implies that 
$\bold{C}_{s}\neq0$ for every $2\le s\le 3m,$ since the surface $\hat{g}$ is substantial.

To prove the claim, we proceed by induction on $s.$ Using that 
$$
d\hat{g}(\overline{E})=\sum\limits_{j=1}^{m}a_{j}dg_{\theta_j}(\overline{E}),
$$ 
the Gauss formula for $g$ and $g_{\theta_j},\,\, j=1,\dots,m,$ yields
\[
\hat{\alpha}_{2}(\overline{E},\overline{E})=\kappa\sum_{j=1}^{m}c_j^2(\xi_{1}
^{\theta_j}+i\xi_2^{\theta_j}),
\]
where $c_j^2=2a_je^{-i\theta_j}.$ Hence, $\bold{C}_2=2T_{\pmb{\theta}}
\bold{C}_1=2T_{\pmb{\theta}}
\bold{a}$ and this proves (\ref{1}) for $s=2.$

Let us assume that (\ref{1})-(\ref{celeosksananea}) hold for any $t\le s.$ We shall prove
that it is also true for $t=s+1.$ From the definition of the higher fundamental forms, we
have that 
\begin{eqnarray} 
\hat{\alpha}_{s+1}(\overline{E},\dots,\overline{E})&=&\left(\overline{\nabla}_
{\overline{E}}\hat{\alpha}_{s}(\overline{E},\dots,\overline{E})\right)^{N_s^
{\hat{g}}}\notag\\
&=&\kappa^{\lambda_{s}}\left(\tilde{\nabla}_{\overline{E}}\left(\frac{1}{\kappa^
{\lambda_{s}}}\hat{\alpha}_{s}(\overline{E},\dots,\overline{E})\right)\right)^{N_s^
{\hat{g}}},\label{a}
\end{eqnarray} 
where $\overline{\nabla}$ is the induced connection of $g^\ast(T\mathbb{S}^{6m-1}),$
$\tilde{\nabla}$ is the induced connection of the induced bundle $(i\circ g)^\ast(T
\mathbb{R}^{6m})$, $i\colon \mathbb{S}^{6m-1}\to\mathbb{R}^{6m}$ is the
standard inclusion, $\lambda_s$ is the exponent of the function $\kappa$ in (\ref{1}) and $(\,\cdot\,)
^{N_s^{\hat{g}}}$ stands for the projection onto the $s$-th normal bundle of $\hat{g}.$
Taking \eqref{apeiroeths} into account, we obtain that 
\begin{equation}\label{naii}
\tilde{\nabla}_{\overline{E}}
\left(\frac{1}{\kappa^{\lambda_s}}\hat{\alpha}_{s}(\overline{E},\dots,\overline{E})
\right)=\sum\limits_{j=1}^{m}c^s_jdg_{\theta_{j}}(\overline{E})
\text{\,\,\,if $s\equiv 0\;\mathrm{mod}\; 6$,}
\end{equation}

\begin{eqnarray}
\tilde{\nabla}_{\overline{E}}\left(\frac{1}{\kappa^{\lambda_{s}}}\hat{\alpha}_{s}
(\overline{E},\dots,\overline{E})\right)
&=\dfrac{1}{2}&\sum\limits_{j=1}^{m}c^s_j\bigg(\langle\nabla_{\overline{E}}
\overline{E},E\rangle dg_{\theta_{j}}(\overline{E})\\ \nonumber
&&+4\kappa e^{-i\theta_{j}}T_{\theta_j}(\xi_{1}+i\xi_{2})\bigg)\text{\,\,if\,\,}s
\equiv 1\;\mathrm{mod}\; 6,
\end{eqnarray}

\begin{equation}
\tilde{\nabla}_{\overline{E}}\left(\frac{1}{\kappa^{\lambda_{s}}}\hat{\alpha}_{s}
(\overline{E},\dots,\overline{E})\right)=\sum\limits_{j=1}^{m}c^{s}_{j}\left(-i\langle
\nabla^\perp_{\overline{E}}\xi_{1},\xi_{2}\rangle T_{\theta_j}(\xi_{1}+i\xi_{2})+
2T_{\theta_j}\xi_{3}\right) \text{\,\,if $s\equiv 2\;\mathrm{mod}\; 6$},
\end{equation}

\begin{equation}
\tilde{\nabla}_{\overline{E}}\left(\frac{1}{\kappa^{\lambda_{s}}}\hat{\alpha}_{s}
(\overline{E},\dots,\overline{E})\right)=-\sum\limits_{j=1}^{m}c^{s}_{j}T_{\theta_j}
(\xi_{1}-i\xi_{2})\text{\,\,\,if $s\equiv 3\;\mathrm{mod}\; 6,$}
\end{equation}

\begin{eqnarray}
\tilde{\nabla}_{\overline{E}}\left(\frac{1}{\kappa^{\lambda_{s}}}\hat{\alpha}_{s}
(\overline{E},\dots,\overline{E})\right)&=&\sum\limits_{j=1}^{m}c^{s}_{j}\Big(-2\kappa
e^{-i\theta_{j}}dg_{\theta_{j}}(E)\\ \nonumber
&&+i\langle\nabla^\perp_{\overline{E}}\xi_{1},\xi_{2}\rangle T_{\theta_j}(\xi_{1}-i
\xi_{2})\Big)\text{\,\,\,if $s\equiv 4\;\mathrm{mod}\; 6$}
\end{eqnarray}
and
\begin{equation}\label{naiii}
\tilde{\nabla}_{\overline{E}}\left(\frac{1}{\kappa^{\lambda_{s}}}\hat{\alpha}_{s}
(\overline{E},\dots,\overline{E})\right)=\dfrac{1}{2}\sum\limits_{j=1}^{m}c^{s}_{j}\left(\langle
\nabla_{\overline{E}}E,\overline{E}\rangle dg_{\theta_{j}}(E)-4g_{\theta_{j}}\right)
\text{\,\,\,if $s\equiv 5\;\mathrm{mod}\; 6$,}
\end{equation}
where $\nabla$ is the Levi-Civit\'a connection on $M.$

Using (\ref{a}) and (\ref{naii})-(\ref{naiii}), after some tedious computations, we obtain 
that (\ref{1}) holds for $t=s+1.$ Taking into account (\ref{1}) for $t=s+1,$ the orthogonality of the
higher normal bundles and (\ref{2})-(\ref{6}) for $t\le s,$ we obtain that (\ref{2})-(\ref{6}) are also true for $t=s+1,$ and this completes the proof of the claim.

From (\ref{si}) and since the length of the semi-axes $\kappa_s$ and $\mu_s$ of the $s$-th curvature ellipse satisfy
$$
\kappa^2_s+\mu^2_s=2^{-2s}\left\Vert \hat{\alpha}_{s+1}(\overline{E},\dots,
\overline{E})\right\Vert^2,
$$ 
we have 
$$
\left\Vert \hat{\alpha}_{s+1}\right\Vert^2=2^{-s}\left\Vert \hat{\alpha}_{s+1}
(\overline{E},\dots,\overline{E})\right\Vert^2.
$$
Clearly (\ref{alphaf}) follows from (\ref{1}) with
\begin{equation*}
\hat{b}_s=
\begin{cases}
\,{2^{(3-4s)/3}}\left\Vert\bold{C}_{s+1}\right\Vert^{2}&\text{ \ if \ }
s\equiv 0\;\mathrm{mod}\; 3,\\[2mm]
\,2^{(1-4s)/3}\left\Vert\bold{C}_{s+1}\right
\Vert^{2}
&\text{ \ if \ }s\equiv 1\;\mathrm{mod}\; 3, \\[2mm]
\,2^{-(1+4s)/3}\left\Vert\bold{C}_{s+1}\right
\Vert^{2}
&\text{ \ if \ }s\equiv 2\;\mathrm{mod}\; 3.
\end{cases}
\end{equation*}

Furthermore, the $s$-th normal curvature is given by
\[
\hat{K}_{s}^\perp=2^{-2s}\left(\left\Vert \hat{\alpha}_{s+1}(\overline{E},\dots,
\overline{E})\right\Vert^4-|\langle \hat{\alpha}_{s+1}(E,\dots,E),\hat{\alpha}_{s+1}
(E,\dots,E)\rangle|^2\right)^{1/2}.
\]
This, combined with (\ref{1}) yields (\ref{curvaturef}), where
\begin{equation*}
\hat{c}_s=
\begin{cases}
\,2^{(3-7s)/3}\left\Vert\bold{C}_{s+1}\right\Vert^{2}&\text{ \ if \ }s\equiv 0\;
\mathrm{mod}\;3,\\[3mm]
\,2^{(1-7s)/3}\left\Vert\bold{C}_{s+1}\right\Vert^{2}
&\text{ \ if \ }s\equiv 1\;\mathrm{mod}\; 3, \\[2mm]
\,2^{-(1+7s)/3}\left(\left\Vert \bold{C}_{s+1}\right\Vert^{4}-|(\bold{C}_{s+1},
\overline{\bold{C}}_{s+1})|^2\right)^{1/2} &\text{ \ if \ }s\equiv 2\;\mathrm{mod}\; 3.
\end{cases}
\end{equation*}

Using (\ref{1}) and the fact that the $s$-th Hopf differential of $\hat{g}$ is written as
$$
\hat{\Phi}_{s}=4^{-(s+1)} \left\langle \hat{\alpha}_{s+1}(E,\dots,E),\hat{\alpha}_{s+1}
(E,\dots,E)\right\rangle\phi^{2s+2},
$$
we obtain that 
\[
\hat{\Phi}_{s}=4^{-(s+1)}\kappa^{2\frac{s+1}{3}}
\sum\limits_{j=1}^{m}(\overline{c}_j^{s+1})^2\phi^{2s+2} \;\;\text{if}\;\; s\equiv2
\;\mathrm{mod}\; 3
\]
and $\hat{\Phi}_{s}=0$  otherwise.

The fact that the second Hopf differential $\Phi$ of $g$ is given by $\Phi=2^{-2}
\kappa^2\phi^6$ completes the proof of part $(iii)$, where 
$\hat{d}_s=2^{-4(s+1)/3}(\overline{\bold{C}}_{s+1},\bold{C}_{s+1}).$ 
Obviously, all Hopf differentials are holomorphic and consequently $\hat{g}$ is exceptional according to
Theorem \ref{ena}.
\end{proof}

In the subsequent lemma, we determine the associated family of any surface 
$\hat{g}=g_{\bold{a},{\pmb{\varphi}}}.$

\begin{lemma}\label{asociatefamilygtheta}
The associated family $\hat{g}_\varphi$ of any minimal surface $\hat{g}=g_{\bold{a},
{\pmb{\theta}}}$ is given by $\hat{g}_\varphi=g_{\bold{a},{\pmb{\varphi}}}$, where
${\pmb{\varphi}}=(\theta_1+\varphi,\dots,\theta_m+\varphi).$
\end{lemma}
\begin{proof}
Let $f\colon M\to\mathbb{S}^{6m-1}$ be the minimal surface
given by $f=g_{\bold{a},{\pmb{\varphi}}}.$ From (\ref{celeosksananea}) we can
easily see that the complex vectors $\bold{C}^{f}_{s},\,\bold{C}_{s}\in\mathbb{C}^m
\smallsetminus\{0\}$ associated to $f$ and $\hat{g}=g_{\bold{a},{\pmb{\theta}}}$,
respectively, satisfy
\[
\bold{C}^{f}_{s}=e^{-i\varphi}\bold{C}_{s} \;\;\text{for any}\;\; 2\le s\le 3m.
\]

Moreover, Proposition \ref{hola}(iii) implies that the $s$-th Hopf differential of $f$ is
given by
\begin{equation*}
\Phi^f_s=
\begin{cases}
\,d^f_s\Phi^{(s+1)/3}&\text{ \,if\, }s\equiv 2\;\mathrm{mod}\; 3, \\[1mm]
\,0&\text{ \,\,otherwise, }\
\end{cases}
\end{equation*}
where $d^f_s=2^{-4(s+1)/3}(\overline{\bold{C}}^f_{s+1},\bold{C}^f_{s+1}).$
Equivalently, we have
\begin{equation*}\label{kanahopf}
\Phi^f_s=
\begin{cases}
\,2^{-4(s+1)/3}e^{2i\varphi}(\overline{\bold{C}}_{s+1},\bold{C}_{s+1})\Phi^{(s+1)/
3}&\text{ \,if\, }s\equiv 2\;\mathrm{mod}\; 3, \\[1mm]
\,0&\text{ \,\,otherwise.}
\end{cases}
\end{equation*}
Thus the  Hopf differentials of $f$ and $\hat{g}$ satisfy
\[
\Phi^f_s= e^{2i\varphi}\hat{\Phi}_s \;\; \text{for any} \;\;1\le s\le 3m-1.
\]
According to \cite[Theorem 5.2]{V}, the associated family of the surface $\hat{g}$ is
$g_{\bold{a},{\pmb{\varphi}}}$ and this completes the proof.
\end{proof}

\section{Exceptional surfaces and the Ricci-like condition}

In this section, we study exceptional surfaces that satisfy the Ricci-like condition $(\ast).$ To prove our main results, we need the following proposition.

\begin{proposition}\label{whhat}
Let $f\colon M\to \mathbb{S}^n$ be a nonflat $r$-exceptional surface
which satisfies the Ricci-like condition $(\ast).$ Then the following hold:

For any $1\leq s\leq r+1,$ we have:
\begin{equation}\label{aexc}
\left\Vert \alpha_{s+1}\right\Vert^2=
\begin{cases}
\,b_s(1-K)^{s/3} &\text{ \,if\, }s\equiv 0\;\mathrm{mod}\; 3, \\[1mm]
\,b_s(1-K)^{(s+2)/3} &\text{ \,if\, }s\equiv 1\;\mathrm{mod}\; 3, \\[1mm]
\,b_s(1-K)^{(s+1)/3} &\text{ \,if\, }s\equiv 2\;\mathrm{mod}\; 3,
\end{cases}
\end{equation}
where $\left\{ b_s\right\}, \left\{ \rho_s\right\}$ are sequences of positive numbers such that $b_1=2, b_{s+1}=\rho_s^2b_s, \rho _s\leq1$ and $\rho _s=1$ if $s\equiv0,1\;\mathrm{mod}\; 3.$ 

Moreover for any $1\leq s\leq r,$ the following hold:
\begin{equation}\label{first}
\Phi_s=0 \text{\,\, if\,\,\,}s\equiv0,1\;\mathrm{mod}\; 3,\\
\end{equation}
\begin{equation}\label{olalazoun}
K_s^*=
\begin{cases}
\,K &\text{ \,if\, }s\equiv 0\;\mathrm{mod}\; 3, \\[1mm]
\,-K&\text{ \,if\, }s\equiv 1\;\mathrm{mod}\; 3, \\[1mm]
\,0 &\text{ \,if\, }s\equiv 2\;\mathrm{mod}\; 3,
\end{cases}
\end{equation}
and
\begin{equation}\label{last}
K_s^\perp=
\begin{cases}
\,c_s(1-K)^{s/3} &\text{ \,if\, }s\equiv 0\;\mathrm{mod}\; 3, \\[1mm]
\,c_s(1-K)^{(s+2)/3} &\text{ \,if\, }s\equiv 1\;\mathrm{mod}\; 3, \\[1mm]
\,c_s(1-K)^{(s+1)/3} &\text{ \,if\, }s\equiv 2\;\mathrm{mod}\; 3,
\end{cases}
\end{equation}
 where $c_s=2^{-s}\rho _sb_s.$ 
\end{proposition}

\begin{proof}
We set $\rho _s=2^sK_s^{\perp}/\left\Vert \alpha_{s+1}\right\Vert^2.$
Since $f$ is $r$-exceptional, the function $\rho _s$ is constant for any $1\leq s\leq r.$
We proceed by induction on $r.$

Assume that $f$ is $1$-exceptional. The Gauss equation implies $\left\Vert
\alpha_2\right\Vert^2=2(1-K).$ Then from Proposition \ref{3i}(i) for $s=1$ and the
Ricci-like condition $(\ast)$, we find $K_1^{\ast}=-K.$ Moreover, we have $K_1^{\perp}=\rho_1(1-K).$ 
We claim that $\rho_1=1.$ Assume to the contrary that
$\rho_1\neq1.$ Then $\Phi_1\neq0$ and Proposition \ref{3i}(ii) combined with the
condition $(\ast)$ yield $K^*_1=K$, which is a contradiction. Hence
$\rho_1=1$ and Proposition \ref{5} yields (\ref{aexc}) for $s=2$ with $b_2=2.$ This
settles the case $r=1.$

Suppose now that (\ref{first})-(\ref{last}) hold if $f$ is $r$-exceptional. We shall prove
that (\ref{first})-(\ref{last}) also hold assuming that $f$ is $(r+1)$-exceptional. By Theorem \ref{ena}, the Hopf differential
$\Phi _{r+1}$ is holomorphic, hence either it is identically zero or its zeros are isolated.

At first we assume that $r\equiv0\;\mathrm{mod}\;3.$ From the inductive assumption,
we have 
$$
\left\Vert \alpha_{r+2}\right\Vert^2=b_{r+1}(1-K)^{(r+3)/3}.
$$
We claim that $\rho_{r+1}=1.$ Arguing indirectly, we assume that $\Phi _{r+1}\neq 0.$
Then Proposition \ref{3i}(iv) yields $K_{r+1}^{\ast}=0.$ Taking into account the 
condition $(\ast)$, Proposition \ref{3i}(ii) implies that $M$ is flat and this is a
contradiction. Thus $\Phi _{r+1}$ is identically zero, or equivalently $\rho _{r+1}=1.$
From Proposition \ref{3i}(iii) and the condition $(\ast)$, we obtain $K_{r+1}^
{\ast}=-K.$ Furthermore, we have 
$$
K_{r+1}^{\perp}=2^{-(r+1)}\rho _{r+1}\left\Vert \alpha_{r+2}\right\Vert^2,
$$
or equivalently 
$$
K_{r+1}^{\perp}=c_{r+1}(1-K)^{(r+3)/3},
$$
with $c_{r+1}=2^{-(r+1)}b_{r+1}.$ Then using Proposition \ref{5}, we obtain 
$$
\left\Vert \alpha_{r+3}\right\Vert^2=b_{r+2}(1-K)^{(r+3)/3},
$$
with $b_{r+2}=b_{r+1}.$

Assume now that $r\equiv1\;\mathrm{mod}\;3.$ From the inductive assumption, we have 
$$
\left\Vert \alpha_{r+2}\right\Vert ^{2}=b_{r+1}(1-K)^{(r+2)/3}.
$$
If \,$\Phi _{r+1}\neq 0,$ then Proposition \ref{3i}(iv) yields $K_{r+1}^{\ast}=0.$ If
$\Phi _{r+1}$ is identically zero, or equivalently $\rho _{r+1}=1$, then Proposition
\ref{3i}(iii) and the condition $(\ast)$ imply that $K_{r+1}^{\ast}=0.$
Furthermore, we have 
$$
K_{r+1}^{\perp}=2^{-(r+1)}\rho _{r+1}\left\Vert \alpha_{r+2}\right\Vert^2,
$$
or equivalently 
$$
K_{r+1}^{\perp}=c_{r+1}(1-K)^{(r+2)/3},
$$
with $c_{r+1}=2^{-(r+1)}\rho_{r+1}b_{r+1}.$ From Proposition \ref{5}, we obtain 
$$
\left\Vert \alpha_{r+3}\right\Vert^2=b_{r+2}(1-K)^{(r+2)/3},
$$
with $b_{r+2}=\rho^2_{r+1}b_{r+1}.$

Finally, we suppose that $r\equiv2\;\mathrm{mod}\;3.$ From the inductive assumption,
we have 
$$
\left\Vert \alpha_{r+2}\right\Vert^2=b_{r+1}(1-K)^{(r+1)/3}.
$$
We claim that $\rho_{r+1}=1.$ Assume to the contrary that $\rho_{r+1}\neq1$ or
equivalently $\Phi _{r+1}\neq 0.$ Then Proposition \ref{3i}(iv) yields $K_{r+1}^{\ast}
=0.$ Taking into account the Ricci-like condition $(\ast),$ Proposition \ref{3i}(ii) implies that
$M$ is flat, which is a contradiction. Hence $\Phi _{r+1}$ is identically zero, or
equivalently $\rho _{r+1}=1.$ From Proposition \ref{3i}(iii) and the condition
$(\ast)$, we obtain $K_{r+1}^{\ast}=K.$ Furthermore, we have
$$
K_{r+1}^{\perp}=2^{-(r+1)}\rho _{r+1}\left\Vert \alpha_{r+2}\right\Vert^2,
$$
or equivalently
$$
K_{r+1}^{\perp}=c_{r+1}(1-K)^{(r+1)/3},
$$
with $c_{r+1}=2^{-(r+1)}\rho_{r+1}b_{r+1}.$ Using Proposition \ref{5}, it follows that 
$$
\left\Vert \alpha_{r+3}\right\Vert^2=b_{r+2}(1-K)^{(r+4)/3}
$$
with $b_{r+2}=b_{r+1}$ and this completes the proof.
\end{proof}

We are ready to prove the main result of this section.
\medskip
\begin{theorem}\label{kanenadenexwlabel}
Let $f\colon M\to\mathbb{S}^n$ be a nonflat simply connected exceptional surface
with substantial odd codimension. If $f$ satisfies the Ricci-like condition $(\ast)$ away
from the isolated points with Gaussian curvature $K=1$, then $n=6m-1$ and there exists
$\bold{a}=(a_1,\dots,a_m)\in\mathbb{S}^{m-1}\subset\mathbb{R}^m$ with $\Pi_{j=1}^{m}a_j\neq0$ and 
${\pmb{\theta}}=(\theta_1,\dots,\theta_m)\in\mathbb{S}^1\times\cdots\times\mathbb{S}^1,$ where 
$0\leq \theta _1<\cdots<\theta_m<\pi$ such that $f=g_{\bold{a},\pmb{\theta}}.$
\end{theorem}

\begin{proof}
We claim that $n\equiv 5\;\mathrm{mod}\;6.$ Arguing indirectly, we suppose at first that
$n=6l+1.$ Since $f$ is $(3l-1)$-exceptional, (\ref{aexc}) yields $\left\Vert\alpha_{3l+1}
\right\Vert ^{2}=b_{3l}(1-K)^{l}.$ Moreover, viewing $f$ as a minimal surface in
$\mathbb{S}^{6l+2},$ we obviously have $K_{3l}^{\perp}=K_{3l}^{\ast}=0.$ Then
from Proposition \ref{3i}(ii), we obtain
$$
\Delta \log\left\Vert \alpha_{3l+1}\right\Vert ^2=2(3l+1)K.
$$
Thus $K=0,$ which is a contradiction.

Suppose that $n=6l+3.$ Since $f$ is $3l$-exceptional, (\ref{aexc}) yields
$$
\left\Vert\alpha_{3l+2}\right\Vert^2=b_{3l+1}(1-K)^{l+1}.
$$
Moreover, viewing $f$ as a minimal surface in $\mathbb{S}^{6l+4},$ we obviously have
$K_{3l+1}^{\perp}=K_{3l+1}^{\ast}=0.$ Then from Proposition \ref{3i}(ii), it follows
that
$$
\Delta \log
\left\Vert \alpha_{3l+2}\right\Vert^2=2(3l+2)K.
$$
Thus $K=0,$ which is a contradiction.

Hence $n\equiv 5\;\mathrm{mod}\;6$ and we may set $n=6m-1.$ According to
(\ref{first}), we have that $\Phi _{r}=0$ if $r\equiv0,1\;\mathrm{mod}\;3.$ Let 
$$
r_{0}=\min \left\{ r:2\leq r\leq 3m-1\;\; \text{with}\;\; \Phi _{r}\neq0\right\}.
$$
Obviously $r_0\equiv2\;\mathrm{mod}\;3.$ Let $z$ be a local complex coordinate
such that the induced metric is given by $ds^2=F|dz|^2.$ From the definition of Hopf
differentials we know that $\Phi _r=f_rdz^{2r+2},$ where $f_r=\langle\alpha_
{r+1}^{(r+1,0)}, \alpha_{r+1}^{(r+1,0)}\rangle.$ 

For any $r\equiv2\;\mathrm{mod}\;3$ and $r\geq r_0$ such that $\Phi_r\neq0,$ we may 
write $\Phi _r=|f_r|e^{i\sigma _r}dz^{2r+2}.$ Using (\ref{aexc}), (\ref{last}) and
(\ref{what}), we obtain 
\begin{equation*}
\Phi _r=2^{-(r+2)}b_rF^{r+1}e^{i\sigma _r}(1-K)^{(r+1)/3}
\left(1-\rho _r^2\right)^{1/2}dz^{2r+2}.
\end{equation*}

We pick a branch $h$ of $f_{r_0}^{3/(r_0+1)}$ and define the form $\Phi =c_0hdz^6,$ 
where $c_0$ is given by
\begin{equation*}
c_0=
\begin{cases}
\,\left(\frac{2}{b_{r_0}\left(1-\rho^2_{r_0}\right)^{1/2}}\right)^{3/(r_0+1)} &\text{if $r_0<3m-1,$}\\[3mm]
\,\left(\frac{2}{b_{r_0}}\right)^{3/(r_0+1)} &\text{if $r_0=3m-1.$}
\end{cases}
\end{equation*}
It is obvious that $\Phi $ is well defined and holomorphic. It follows that
\begin{equation*}
\Phi _{r}=\frac{1}{2}b_r\left(1-\rho _r^2\right)^{1/2}e^{i\big(
\sigma _r-\frac{r+1}{r_0+1}\sigma _{r_0}\big)}\Phi ^{(r+1)/3}.
\end{equation*}
From the holomorphicity of $\Phi _r$ and $\Phi,$ we deduce that $\sigma
_r-\frac{r+1}{r_0+1}\sigma _{r_0}$ is constant. Moreover, we easily
see that 
$$
|c_0h|^2=\left(\frac{F}{2}\right)^6\left(1-K\right)^2.
$$ 

Using \cite[Theorem 11.1]
{EschVl}, we infer that\ there exists a pseudoholomorphic curve
$g\colon M\to\mathbb{S}^5$ whose second Hopf differential is $\Phi.$ 

We consider a surface $\hat{g}=g_{\bold{a},\pmb{\theta}}$ which according to
Proposition \ref{hola} is exceptional for any $\bold{a}\in\mathbb{S}^{m-1}$ and
$\pmb{\theta}.$ Setting $\hat{\rho}_s=2^s\hat{K}_s^\perp/\left\Vert\hat{\alpha}_{s+1}\right\Vert^2,$ it follows from Proposition \ref{hola} that
\begin{equation*}
\hat{\rho}_s=
\begin{cases}
\,\left(1-\frac{\left|\left(\bold{C}_{s+1},\overline{\bold{C}}_{s+1}\right)\right|^2}
{\left\Vert \bold{C}_{s+1}\right\Vert^4}\right)^{1/2}
&\text{ \ if \ }s\equiv 2\;\mathrm{mod}\; 3,\\[4mm]
\,\,\,1&\text{ \ otherwise. \ }
\end{cases}
\end{equation*}

We now claim that we can choose $\bold{a}\in\mathbb{S}^{m-1}$ and $\pmb{\theta}$
such that
\[
b_r=\hat{b}_r,\,\, c_r=\hat{c}_r \;\;\text{and}\;\; \rho_r=\hat{\rho}_r \;\;\text{for every}\;\;
1\leq r\leq 6m-3,
\]
where $\hat{b}_r, \hat{c}_r, b_r, c_r $ and $\rho_r$ are the sequences in Proposition
\ref{hola} and Proposition \ref{whhat}, respectively. Obviously, Proposition \ref{whhat}
gives that
\[
b_r=\hat{b}_r=2 \;\;\text{for}\;\;  r=1,2,\;\; c_1=\hat{c}_1=1 \;\;\text{and}\;\; \rho_1=
\hat{\rho}_1=1.
\]
We choose $\bold{a}$ and $\pmb{\theta}$ such that the unitary transformation
$T_{2\pmb{\theta}}$ satisfies
\[
\left|\left(T_{2\pmb{\theta}}\bold{a},\bold{a}\right)\right|^2=1-\rho^2_2.
\]

According to (\ref{celeosksananea}), the above is equivalent to $\rho_2=\hat{\rho}_2.$
Then using Proposition \ref{whhat}, we obtain that 
$$
b_{r+1}=\hat{b}_{r+1},\, c_r=\hat{c}_r \;\;\text{and}\;\; \rho_r=\hat{\rho}_r
\;\;\text{for}\;\; 1\le r\le4.
$$
Similarly, we may choose $\bold{a}$ and $\pmb{\theta}$ such that 
\[
\frac{\left|\left(T_{2\pmb{\theta}}\bold{C}_4,\bold{C}_4\right)\right|^2}{\|\bold{C}_4
\|^4}=1-\rho_5^2,
\]
or equivalently $\rho_5=\hat{\rho}_5,$ according to (\ref{celeosksananea}).
Repeating this argument, and choosing $\bold{a}$ and $\pmb{\theta}$ such that
$\rho_r=\hat{\rho}_r$ for any $r\equiv 2\;\mathrm{mod}\; 3$, the claim follows
inductively.

Thus, Proposition \ref{whhat} implies that the $a$-invariants of the minimal surface $f$
coincide with those of $\hat{g}=g_{\bold{a}, \pmb{\theta}}$ for appropriate $\bold{a}$
and $\pmb{\theta}.$

It follows from \cite[Theorem 5.2]{V} that $f$ is a member of the associate
family of $\hat{g}$, which in view of Lemma \ref{asociatefamilygtheta} completes the
proof.
\end{proof}
For the proof of Theorem \ref{artiadimension} below, we recall the following well known
lemma.

\begin{lemma}\label{QPgrad}
Let $M$ be a two-dimensional Riemannian manifold and let $f\colon M\to
\mathbb{R}$ be a smooth function such that $\Delta f=P(f)$ and $\left\Vert\nabla
f\right\Vert^{2}=Q(f)$ for smooth functions $P,Q\colon\mathbb{R}\to \mathbb{R},$
where $\nabla f\ $denotes the gradient of $f.$ Then on $\left\{ p\in
M:\nabla f(p)\neq 0\right\},$ the Gaussian curvature $K$ satisfies 
\begin{equation*}
2KQ+(2P-Q^{\prime})(P-Q^{\prime })+Q(2P^{\prime }-Q^{\prime \prime })=0.
\end{equation*}
\end{lemma}

For minimal surfaces  in substantial even codimension, we prove the following result.

\begin{theorem}\label{artiadimension}
(i) Substantial exceptional surfaces in $\mathbb{S}^{6m}$ cannot satisfy the Ricci-like condition $(\ast).$

(ii) Substantial $\left[\frac{n-1}{2}\right]$-exceptional surfaces in an
even dimensional sphere $\mathbb{S}^n$ cannot satisfy the Ricci-like condition
$(\ast).$
\end{theorem} 

\begin{proof}
(i) Assume to the contrary that $f\colon M \to \mathbb{S}^{6m}$ is a substantial 
exceptional surface that satisfies the condition $(\ast).$ Since $f$ is $(3m-2)$-exceptional, Proposition \ref{whhat} yields
$\left\Vert\alpha_{3m}\right\Vert ^2=b_{3m-1}(1-K)^m$ and $K_{3m-2}^{\ast}
=-K.$ Moreover, combining Proposition \ref{5} with Proposition \ref{whhat}, we
find that
\begin{equation*}
K_{3m-1}^{\ast}=\frac{2^{3m-1}K_{3m-1}^{\perp}}{b_{3m-1}(1-K)^m}.
\end{equation*}

By Theorem \ref{ena}, $\Phi _{3m-1}$ is holomorphic. Hence either it is identically
zero or its zeros are isolated. If $\Phi _{3m-1}$ is identically zero, then $
f$ is $(3m-1)$-exceptional, and (\ref{olalazoun}) yields $K_{3m-1}^{\ast}=0.$
Then the above equation implies $K_{3m-1}^{\perp}=0.$ This means that  $f$ lies in a totally geodesic $\mathbb{S}^{6m-1}$ of $\mathbb{S}^{6m}$ (cf. 
\cite[p. 96]{O}), which is a contradiction. Suppose now that $\Phi _{3m-1}\neq 0.$ By virtue
of Proposition \ref{3i}(ii), we have
\begin{eqnarray*}
\Delta \log \left(\left\Vert \alpha_{3m}\right\Vert^2
+2^{3m-1}K_{3m-1}^{\perp}\right) &=&2\big(3mK-K_{3m-1}^{\ast}\big), \\
\Delta \log \left(\left\Vert \alpha_{3m}\right\Vert
^2-2^{3m-1}K_{3m-1}^{\perp}\right) &=&2\big(3mK+K_{3m-1}^{\ast}\big).
\end{eqnarray*}

Using the condition $(\ast)$ and setting 
$\rho=2^{3m-1}K_{3m-1}^{\perp}/\left\Vert \alpha_{3m}\right\Vert ^2,$ 
the above equations are equivalent to
\begin{equation}\label{paliapothnarxh}
\Delta \log \left(1+\rho \right) =-2K_{3m-1}^{\ast} \;\; \text{and} \;\; \Delta
\log \left(1-\rho \right) =2K_{3m-1}^{\ast}.
\end{equation}
Since $\rho = 2^{3m-1}K_{3m-1}^{\perp}/\left\Vert \alpha_{3m}\right\Vert^2,$ we
obtain $K_{3m-1}^{\ast}=\rho.$

Then equations (\ref{paliapothnarxh}) are written equivalently
$$
\Delta \rho =-2\rho (1+\rho^2)
\;\; \text{and} \;\; \left\Vert\nabla \rho \right\Vert^2=2\rho^2(1-\rho^2).
$$ 
If the function $\rho$ is constant, then $\rho =0$ and consequently $K_{3m-1}^{\perp}=0,$ which contradicts the fact
that $f$ is substantial. If $\rho$ is not constant, then Lemma \ref{QPgrad}
yields $K=-8,$ which contradicts the Ricci-like condition $(\ast).$

(ii) Assume that $f\colon M\to \mathbb{S}^n$ is a substantial $[(n-1)/2]$-exceptional surface which satisfies the condition $(\ast),$ where $n$
is even. It suffices to consider the case $n=6m+2$ and $n=6m+4,$ since the case
$n=6m$ was settled in (i).

At first let us suppose that $n=6m+2.$ Since $f$ is $3m$-exceptional, (\ref{first}) and 
(\ref{olalazoun}) yield \,$\Phi _{3m}=0$ and\, $K_{3m}^{\ast}=K.$ By virtue of
Proposition \ref{5}, we obtain 
\begin{equation*}
K_{3m}^{\ast}=\frac{K_{3m}^{\perp}\left\Vert \alpha_{3m}\right\Vert^2}{
2^{3m-2}\left(K_{3m-1}^{\perp}\right)^2}.
\end{equation*}
Then, using (\ref{aexc}) and (\ref{last}), we have that $K_{3m}^{\ast}=1$, which
is a contradiction.

We suppose now that $n=6m+4.$ Since $f$ is $(3m+1)$-exceptional, (\ref{first}) and 
(\ref{olalazoun}) yield \,$\Phi _{3m+1}=0$ and\, $K_{3m+1}^{\ast}=-K.$ From
Proposition \ref{5} it follows that
\begin{equation*}
K_{3m+1}^{\ast}=\frac{K_{3m+1}^{\perp}\left\Vert \alpha_{3m+1}\right\Vert^2}
{2^{3m-1}\left(K_{3m}^{\perp}\right)^2}.
\end{equation*}
Using (\ref{aexc}) and (\ref{last}), we find that $K_{3m+1}^{\ast}=1-K$, which is a
contradiction, and this completes the proof.
\end{proof}

\section{Global results}

In this section, we prove results for compact minimal surfaces that satisfy the  condition $(\ast)$ and are not homeomorphic to the torus. We recall from Lemma \ref{avtf}, that such surfaces cannot be homeomorphic to the sphere $\mathbb S^2.$

\begin{theorem}\label{pamegiallaksana}
Let $f\colon M\to \mathbb{S}^n$ be a compact substantial minimal surface with  
genus $g\geq 2$ which satisfies the Ricci-like condition $(\ast)$ away
from isolated points where the Gaussian curvature satisfies $K=1.$ If the eccentricity
$\varepsilon _r$ of the higher curvature ellipses of order 
$r\equiv 0\;\mathrm{mod}\; 3$ for any $1\leq r\leq s$ satisfies the condition
\begin{equation*}
\int_{M}\frac{\varepsilon_r}{\left(1-K\right)^{\gamma}}dA<\infty 
\end{equation*}
for some constant $\gamma \geq 4/3,$ then $f$ is $s$-exceptional.
\end{theorem}

\begin{proof}
According to Lemma \ref{avtf}, the function $1-K$ is of absolute value type with
nonempty zero set $M_0=\left\{ p_1,\dots,p_m\right\} $ and corresponding order $\mathrm{ord}_{p_j}(1-K)=2k_j.$ For each point $p_j, j=1,\dots,m,$ we choose a
local complex coordinate $z$ such that $p_j$ corresponds to $z=0$\ and the induced
metric is written as $ds^2=F|dz|^2.$ Around\ $p_j,$ we have that
\begin{equation}\label{den}
1-K=|z|^{2k_j}u_0,
\end{equation}
where $u_0$ is a smooth positive function.

We shall prove that $f$ is $s$-exceptional by induction. At first we show that $f$ is
$1$-exceptional. In fact, we can prove that $f$ is $1$-isotropic. We know that the first
Hopf differential $\Phi _1=f_1dz^4$ is holomorphic. Hence either $\Phi _1$ is identically zero, or its zeros are isolated.  Assume now
that $\Phi _1$ is not identically zero. Obviously, $\Phi _1$ vanishes at each $p_j.$
Thus we may write $f_1=z^{l_1(p_j)}\psi _1$\ around\ $p_j,$ where
$l_1(p_j)$\ is the order of $\Phi_1$\ at $p_j,$ and $\psi _1$ is a nonzero
holomorphic function. Bearing in mind (\ref{what}),\ we obtain 
\begin{equation*}
\frac{1}{4}\left\Vert \alpha_2\right\Vert^4-(K_1^{\perp})^2=2^4F^{-4}|\psi _1|^2|z|^{2l_{1}(p_j)}
\end{equation*}
around $p_j.$ In view of (\ref{den}) and the fact that $\left\Vert
 \alpha_2\right\Vert^2=2(1-K),$ we find that the function $u_1\colon M
\smallsetminus M_0\to \mathbb{R}$ defined by
$$
u_1=\frac{\left((1-K)^2-(K_1^{\perp})^2\right)^3}{(1-K)^4},
$$
around $p_j,$ is written as 
\begin{equation}\label{tora}
u_1=2^{12}F^{-12}u_0^{-4}|\psi _1|^6|z|^{6l_{1}(p_j)-8k_j}.
\end{equation}

Since $
u_1\leq (1-K)^2,$ from (\ref{tora}) we deduce that $l_1(p_j)\geq 2k_j$ and we
can extend $u_1$ to a smooth function on $M.$ It follows from Proposition \ref{3i}(ii)
and the condition $(\ast)$ that $\log u_1$ is harmonic away from the isolated zeros of $u_1$. By continuity, the function  $u_1$ is subharmonic everywhere on
$M.$ Using the maximum principle, we deduce that $u_1$ is a positive constant. This
contradicts the fact that $K=1$ on $M_0.$

Suppose now that $f$ is $(r-1)$-exceptional for $r\ge2.$ We note that $M$ cannot be flat
due to our assumption on the genus. We shall prove that $f$ is also $r$
-exceptional. From \cite[Proposition 4]{V08}, we know that $\Phi _r=f_rdz^{2r+2}$
is globally defined and holomorphic. Hence either $\Phi _r=0$ or its zeros are isolated.
In the former case, $f$ is $r$-exceptional.

Assume now that $\Phi _r$ is not identically zero. Obviously, $\Phi _r$ vanishes at
$p_j.$ Hence we may write $f_r=z^{l_r(p_j)}\psi _r$\ around\ $p_j,$
where $l_r(p_j)$\ is the order of $\Phi _r$\ at $p_j,$ and $\psi _r$ is a
nonzero holomorphic function. Bearing in mind (\ref{what}),\ we obtain 
\begin{equation}\label{aque}
\left\Vert \alpha_{r+1}\right\Vert^4-4^r(K_r^{\perp})^2=4^{r+2}F^{-2(r+1)}|\psi _r|^2|z|^{2l_r(p_j)}
\end{equation}
around $p_j.$ In view of (\ref{den}), we find that 
\begin{equation}\label{utora}
u_r=4^{3(r+2)}F^{-6(r+1)}u_0^{-2(r+1)}|\psi
_r|^6|z|^{6l_r(p_j)-4k_j(r+1)},
\end{equation}
where $u_{r}\colon M\smallsetminus M_0\to \mathbb{R}$ is the smooth function
(see Proposition \ref{neoksanaafththfora})
given by 
$$
u_r=\frac{\left(\left\Vert \alpha_{r+1}\right\Vert^4-4^r(K_r^{\perp})^2\right)^3}{(1-K)^{2(r+1)}}.
$$

We claim that $r\equiv 2\;\mathrm{mod}\; 3.$ Arguing indirectly, we at first assume that
$r\equiv 0\;\mathrm{mod}\; 3.$ Since $\varepsilon_r^2/(2-
\varepsilon_r^2)\leq \varepsilon_r,$ our assumption implies 
\begin{equation*}
\int_{M}\frac{\varepsilon_r^2}{(2-\varepsilon_r^2)\left(1-K\right)^{\gamma}}
dA<\infty,
\end{equation*}
or equivalently, bearing in mind (\ref{elipsi}) and (\ref{si}),
\begin{equation*}
\int_{M}\frac{\left(\left\Vert \alpha_{r+1}\right\Vert^4-4^r(K_r^{\perp})^2\right)^{1/2}}{\left(1-K\right)^{\gamma}\left\Vert \alpha_{r+1}\right\Vert^2}dA<\infty.
\end{equation*}
Taking into account (\ref{aexc}), the above becomes 
\begin{equation*}
\int_{M}\frac{\left(\left\Vert \alpha_{r+1}\right\Vert^4-4^r(K_r^{\perp})^2\right)^{1/2}}{\left(1-K\right)^{\gamma +\frac{r}{3}}}dA<\infty.
\end{equation*}
We consider the subset
$$
U_{\delta}(p_j)=\left\{p\in M: |z(p)|<\delta \right\}, j=1,\dots,m.
$$
Using (\ref{den}) and (\ref{aque}), the above inequality implies that
\begin{equation*}
\int_{U_{\delta _0}(p_j)\smallsetminus U_{\delta}(p_j)}|z|^{l_r(p_j)-2k_j(\gamma+\frac{r}{3})}dA<c
\end{equation*}
for any $\delta <\delta _0,$ where $c$ is a positive constant and\ $\delta_0$ is small
enough. We set $z=\rho e^{i\theta}.$ Since $dA=F\rho d\rho\wedge d\theta,$ we
deduce that
$$
\int_0^{\delta _0}\rho^{l_r(p_j)-2k_j(\gamma+\frac{r}{3})+1}d\rho <\infty.
$$
This implies that
$$
l_r(p_j)>2k_j(\gamma+\frac{r}{3})-2.
$$
Summing up, we obtain
$$
N(\Phi_r)+2m>2(\gamma+\frac{r}{3})\sum_{j=1}^{m}k_j.
$$
Using Lemma \ref{forglobal}(ii) and \eqref{denginetai} in Lemma \ref{avtf}, it follows that
\begin{equation*}
\chi (M)(3\gamma-1)+m>0.
\end{equation*}
On the other hand, \eqref{denginetai} implies that $m\leq -3\chi (M)$, which contradicts the above and the hypothesis that $\chi (M)<0.$

Now assume that $r\equiv 1\;\mathrm{mod}\;3.$ Bearing in mind (\ref{aexc}), we deduce that $u_r\leq
b_r^6(1-K)^2.$ Using (\ref{den}) and (\ref{utora}), we obtain $3l_r(p_j)
\ge2k_j(r+2).$ Then  from Lemma \ref{avtf}, we conclude that
$$
N(\Phi_r)\ge-2(r+2)\chi (M).
$$
Due to  Lemma \ref{forglobal}(ii), the above contradicts our hypothesis on the genus.

Therefore, we conclude that $r\equiv2\;\mathrm{mod}\; 3.$ By virtue of (\ref{aexc}), we obtain
$u_r\leq b_r^6.$ Then (\ref{utora}) implies $3l_r(p_j)\geq 2k_j(r+1),$ and
we can extend $u_r$ to a smooth function on $M.$ It follows from Proposition
\ref{3i}(ii) and the Ricci-like condition $(\ast)$ that $\log u_r$ is harmonic away from the
zeros which are isolated, and consequently by continuity $u_r$ is subharmonic
everywhere on $M.$ By the maximum principle, we deduce that the function $u_r$ is a positive
constant. This shows that the $r$-th curvature ellipse has constant eccentricity, i.e., the surface $f$ is
$r$-exceptional. This completes the proof.
\end{proof}

For compact minimal submanifolds in spheres with low codimension, we prove the following result.

\begin{corollary}
Let $f\colon M\to\mathbb{S}^n$ be a substantial minimal surface with $4\le n\le7.$
If $M$ is compact and not homeomorphic to the torus, then it cannot be locally isometric
to a pseudoholomorphic curve in $\mathbb{S}^5$, unless $n=5.$
\end{corollary}

\begin{proof}
From Lemma \ref{avtf}, we have that the genus of $M$ satisfies $g\geq2.$ We assume that $n\neq5.$ For $n=4$ and $n=6,$ the result follows immediately from
Theorem \ref{pamegiallaksana} and Theorem \ref{artiadimension}(ii). In the case where
$n=7$, Theorem \ref{pamegiallaksana} implies that the surface is exceptional and the
result follows from Theorem \ref{kanenadenexwlabel}.
\end{proof}

\begin{remark}
The assumption in Theorem \ref{pamegiallaksana} on the eccentricity of curvature ellipses of order 
$r\equiv 0\;\mathrm{mod}\; 3$ could be replaced by the condition
$$
\varepsilon_r\leq (1-K)^\beta
$$
for positive constants $c$ and $\beta>1/3.$ Both conditions claim that the curvature ellipses of order 
$r\equiv 0\;\mathrm{mod}\; 3$ tend to be circles  close to totally geodesic points. We don't know whether Theorem \ref{pamegiallaksana} holds without this assumption in any codimension.

\end{remark}

The following global result is complementary to Theorem \ref{artiadimension}.

\begin{theorem}
Let $f\colon M\to\mathbb{S}^{6m+4}, m\geq1,$ be a substantial exceptional surface. If $M$ is compact with genus 
$g\geq 2,$ then it cannot be locally isometric to a pseudoholomorphic curve in $\mathbb{S}^5.$
\end{theorem}
\begin{proof}
We assume to the contrary that the surface satisfies the Ricci-like condition ($\ast$). Since
$f$ is $3m$-exceptional, from Proposition \ref{neoksanaafththfora} we know that the
Hopf differential $\Phi _{3m+1}=f_{3m+1}dz^{6m+4}$ is globally defined and
holomorphic. Hence either $\Phi _{3m+1}=0$ or its zeros are isolated.

Theorem \ref{artiadimension}(ii) implies that the Hopf differential $\Phi _{3m+1}$
cannot vanish identically.

According to Lemma \ref{avtf}, the function $1-K$ is of absolute value type with
nonempty zero set $M_0=\left\{ p_1,\dots,p_m\right\} $ and corresponding order
$\mathrm{ord}_{p_j}(1-K)=2k_j.$ For each point $p_j, j=1,\dots, m,$ we choose a local complex
coordinate $z$ such that $p_j$ corresponds to $z=0$\ and the induced metric is written
as $ds^2=F|dz|^2.$ Around\ $p_j,$ we have that
\begin{equation}\label{ksanaidio}
1-K=|z|^{2k_j}u_0,
\end{equation}
where $u_0$ is a smooth positive function.

Obviously, $\Phi _{3m+1}$ vanishes at $p_{j}.$ Hence we may write $f_{3m+1}=
z^{l(p_{j})}\psi$ around\ $p_{j},$ where $l(p_{j})$\ is the order of $\Phi _{3m+1}$\
at $p_{j},$ and $\psi$ is a nonzero holomorphic function. Bearing in mind (\ref{what}),\
we obtain 
\begin{equation*}\label{aque2}
\left\Vert \alpha_{3m+2}\right\Vert^4-4^{3m+1}(K_{3m+1}^{\perp})^2=2^{6(m+1)}F^{-2(3m+2)}|\psi|^2|z|^{2l(p_j)}
\end{equation*}
around $p_j.$ In view of (\ref{ksanaidio}), we find that 
\begin{equation}\label{utoratora}
u=2^{18(m+1)}F^{-6(3m+2)}u_0^{-2(3m+2)}|\psi
|^6|z|^{6l(p_j)-4k_j(3m+2)},
\end{equation}
where $u\colon M\smallsetminus M_0\to \mathbb{R}$ is the smooth function
(see Proposition \ref{neoksanaafththfora}) given by 
$$
u=\frac{\left(\left\Vert \alpha_{3m+2}\right\Vert^4-4^{3m+1}(K_{3m+1}^{\perp})^2\right)^3}{(1-K)^{2(3m+2)}}.
$$
Using (\ref{aexc}), it follows that $u\leq
b_{3m+1}^6(1-K)^2.$ Then (\ref{ksanaidio}) and (\ref{utoratora}) imply that
$l(p_j)\ge2k_j(m+1).$ By Lemma \ref{avtf}, we deduce that
$$
N(\Phi_{3m+1})\ge-6(m+1)\chi (M).
$$
It follows from Lemma \ref{forglobal}(ii) that the above contradicts our hypothesis on the genus and the theorem is proved.
\end{proof}

\end{document}